\documentclass[11pt,a4paper]{article}
\usepackage{a4wide}
\usepackage{amsfonts,amsmath,amsthm,amssymb}
\usepackage{amstext}
\usepackage{graphicx,geompsfi,subfigure}
\newcommand{\e}{\varepsilon}

\makeatletter
\@addtoreset{equation}{section}
\makeatother
\newtheorem{lemma}{Lemma}[section]
\newtheorem{theorem}[lemma]{Theorem}
 
\newtheorem{corollary}[lemma]{Corollary} 
\newtheorem{conjecture}[lemma]{Conjecture}

\newcommand{\weak}{\rightharpoonup}
\let\weakto\weak

\DeclareMathOperator*\wliminf{w-liminf}
\newcommand{\beq}{\begin{equation}}
\newcommand{\eeq}{\end{equation}}

\def\0{\mathbf 0}
\DeclareMathOperator{\dist}{dist}
\long\def\drop#1{}
\def\pref#1{(\ref{#1})}
\DeclareMathOperator{\supp}{supp}
\DeclareMathOperator{\diam}{diam}

\def\R{\mathbf R}
\def\N{\mathbf N}
\def\Z{\mathbf Z}
\def\H{\mathcal H}
\def\Hmo#1{\|#1-{\textstyle\dashint #1}\|_{H^{-1}}}
\def\Hmospace#1#2{\|#1-{\textstyle\dashint #1}\|_{H^{-1}(#2)}}
\def\Hmospacea#1#2{\left\|#1-{\textstyle\dashint_{\Ta} #1}\right\|_{H^{-1}(#2)}}
\def\Hmospaceb#1#2{\left\|#1-{\textstyle\dashint_{\Tb} #1}\right\|_{H^{-1}(#2)}}

\def\Fea{{{\sf{E_\eta^{2d}}}}}
\def\Feb{{\sf{E_\eta^{3d}}}}
\def\Fza{{\sf{E_0^{2d}}}}
\def\Fzb{{\sf{E_0^{3d}}}}

\def\fza{{e_0^{\rm 2d}}}
\def\fzb{{e_0^{\rm 3d}}}

\def\lscfza{\overline{e_0^{\rm 2d}}}
\def\Hea{{\sf{F_\eta^{2d}}}}
\def\Heb{{\sf{F_\eta^{3d}}}}

\def\Hza{{\sf{F_0^{2d}}}}
\def\Hzb{{\sf{F_0^{3d}}}}

\def\T{\mathbf{T}}
\def\Ta{{\mathbf{T}^2}}
\def\Tb{{\mathbf{T}^3}}
\def\logeta{\left|\log\eta\right|}
\def\invlogeta{\left|\log\eta\right|^{-1}}
\def\L{\mathcal L}
\def\Xint#1{\mathchoice
   {\XXint\displaystyle\textstyle{#1}}%
   {\XXint\textstyle\scriptstyle{#1}}%
   {\XXint\scriptstyle\scriptscriptstyle{#1}}%
   {\XXint\scriptscriptstyle\scriptscriptstyle{#1}}%
   \!\int}
\def\XXint#1#2#3{{\setbox0=\hbox{$#1{#2#3}{\int}$}
     \vcenter{\hbox{$#2#3$}}\kern-.5\wd0}}

\def\dashint{\Xint-}
\def\tdashint{{\textstyle\dashint}}
\newenvironment{remark}%
  {\par\medbreak\refstepcounter{lemma}%
    \noindent\textbf{Remark~\thetheorem. }}%
  {\par\medskip}

\begin{document}

\begin{title}{Small Volume Fraction Limit of the Diblock Copolymer Problem: I. Sharp Interface Functional}

\author{Rustum Choksi\footnote{Department of Mathematics, Simon Fraser University, 
Burnaby, Canada, choksi@math.sfu.ca}  \and Mark A. Peletier\footnote{Department of Mathematics and Institute for Complex Molecular Systems, Technische Universiteit Eindhoven, The Netherlands, m.a.peletier@tue.nl}}
\end{title}

\maketitle 

\begin{abstract} 
We present the first of two articles on the small volume fraction limit of a nonlocal Cahn-Hilliard functional introduced  to model microphase separation of diblock copolymers. Here we focus attention on the sharp-interface version of the functional and 
consider a limit in which  the volume fraction tends to zero but the number of minority phases (called {\it particles}) remains $O(1)$. 
Using the language of $\Gamma$-convergence, we focus on   two levels of this convergence, and derive first and second order {\it effective}  energies, whose energy landscapes are simpler and more transparent. These limiting energies are only finite on weighted sums of delta functions, corresponding to the concentration of mass into `point particles'.
At the highest 
level, the effective energy is entirely local and contains information about the structure of each particle but no information about their spatial distribution. 
 At the next level we encounter a Coulomb-like 
interaction between the particles, which is responsible for the pattern formation. 
We present the results here in both  three and two dimensions.


\medskip


\end{abstract}

\tableofcontents
\section{Introduction}

This paper and its companion paper \cite{CP2}  are concerned with asymptotic properties of  
 two energy functionals. In either case, the order parameter $u$ is defined on the flat torus $\T^n=\R^n/\Z^n$, i.e.\ the square $[-\frac{1}{2},\frac{1}{2}]^n$ with periodic boundary conditions, and has two preferred states $u = 0$ and $u = 1$.
We will be concerned  with  both $n=2$ and $n = 3$.
The nonlocal Cahn-Hilliard functional is defined on $H^1(\R^n)$ and is given by
\begin{equation}
\label{def:Ees}
{\cal E}^\e (u) \;:=\; \e \,  \int_{\T^n} \,  |\nabla u|^2  \, d { x} \, \, + \,\, \frac1\e\int_{\T^n} u^2(1-u^2) \, d {x}
 \, \, + \,\, \sigma
 \, \Hmospace{u}{\T^n}^2.
\end{equation} 
Its sharp interface limit (in the sense of $\Gamma$-convergence),  defined on $BV(\T^n; \{0, 1\})$ (characteristic functions of finite perimeter), is given by~\cite{RW0}
\begin{equation}
\label{def:Fes}
{\cal E} (u) \;:=\;  \,   \int_{\T^n} \,  |\nabla u| 
 \, \, + \,\, \gamma
 \, \Hmospace{u}{\T^n}^2. 
\end{equation}
In both cases we wish to explore the behavior of these functionals, including the structure of their minimizers, in the limit of small volume fraction  $\dashint_{\T^n} u$.  The present article addresses the sharp interface functional (\ref{def:Fes}); the diffuse-interface functional $\cal E^\e$ is treated in the companion article~\cite{CP2}.

\subsection{The diblock copolymer problem}

The minimization of these nonlocal perturbations of standard perimeter problems are natural model problems  for pattern formation induced by competing short and long-range interactions~\cite{SA}. However, these energies have been introduced to the mathematics literature  because of their connection to a model for microphase separation of diblock copolymers~\cite{BF}. 

A diblock copolymer is a linear-chain molecule consisting of two sub-chains joined 
covalently 
to each other. 
One of the sub-chains is made of  $N_A$ monomers of type A and the other consists of $N_B$ monomers of type B. 
Below a critical temperature,  
even a weak repulsion between unlike monomers A and B induces a strong repulsion between 
the 
sub-chains, causing the
sub-chains to segregate. A macroscopic segregation where the sub-chains detach from one 
another  
cannot occur 
because the chains are chemically bonded. Rather, 
a phase separation on a mesoscopic scale with A and B-rich domains emerges. 
Depending on the material properties of the diblock macromolecules, the observed mesoscopic domains 
are highly regular periodic structures including lamellae, spheres, 
cylindrical tubes, and  double-gyroids (see for example \cite{BF}).  

A density-functional theory, 
first proposed by Ohta and Kawasaki \cite{OK}, 
gives rise to a \emph{nonlocal} free energy~\cite{NO} in which  the Cahn-Hilliard free energy is augmented by a long-range interaction term, which is associated with the
connectivity of the sub-chains in the diblock copolymer  
macromolecule:\footnote{See \cite{CR} for a derivation and  the relationship to the physical material parameters and basic models for inhomogeneous polymers. Usually the wells are taken to be $\pm 1$ representing pure phases of $A$ and $B$-rich regions. For convenience, we have rescaled to wells at $0$ and $1$.}
 \begin{equation}
 \label{OK}
 \frac{\e^2}{2} \,  \int_{\T^n} \,  |\nabla u|^2  \, d { x} \, \, + \,\,
 \int_{\T^n} \, 
u^2(1 - u^2)  \, d {x}
 \, \, + \,\, \frac{\sigma}{2} \,\, \| u - M \|_{H^{-1}(\T^n)}^2.
 \end{equation}
Often this energy is minimized under a mass or volume constraint
 \begin{equation}\label{mass-constraint}
  \dashint_{\T^n} u \, = \, M.
   \end{equation}
Here $u$ represents the relative monomer density, with $u = 0$ corresponding to a pure-$A$ region and $u = 1$ to a pure-$B$ region; the interpretation of $M$ is therefore the relative abundance of the $A$-parts of the molecules, or equivalently the volume fraction of the $A$-region. The constraint of fixed average $M$ reflects that in an experiment the composition of the molecules is part of the preparation and does not change during the course of the experiment. 
In~\pref{OK} the incentive for pattern formation is clear: the first term penalizes oscillation, the second term favors separation into regions of $u=0$ and $u=1$, and the third favors rapid oscillation. Under the mass constraint~\pref{mass-constraint} the three can not vanish simultaneously, and the net effect is to set a fine scale structure depending on $\e, \sigma$ and $M$.  Functional (\ref{def:Ees})  is simply a rescaled version of~\pref{OK} with the 
choice of $\sigma = \e \gamma$.  
Its sharp-interface (strong-segregation) limit, in the sense of $\Gamma$-convergence, is then given by 
(\ref{def:Fes}) \cite{RW0}.   

\subsection{Small volume fraction regime of the diblock copolymer problem}

The precise geometry of the phase distributions (i.e.\ the information contained in a minimizer of~\pref{OK}) depends largely on the volume fraction $M$. In fact, as explained in \cite{CPW}, the two natural parameters controlling the phase diagram are $\e \sqrt{\sigma}$ and $M$. When the combination  $\e \sqrt{\sigma}$ is small and $M$ is close to $0$ or $1$, numerical experiments~\cite{CPW} and experimental observations~\cite{BF} 
reveal structures resembling {\it small well-separated spherical regions of the minority phase}. We often refer to  such small regions as {\it particles}, and they are the central objects of study of this paper. 

Since we are interested in a regime of small volume fraction, it seems natural to seek asymptotic results.  It is the purpose of this article and its companion article \cite{CP2} to give a rigorous 
asymptotic description of the energy in a limit wherein the volume fraction tends to zero but where the number of particles in a minimizer remains $O(1)$. That is, we examine the limit where minimizers converge to weighted Dirac delta point measures and seek effective energetic descriptions for their positioning and local structure. 
Physically, our regime corresponds to diblock copolymers of very small molecular weight (ratio of $B$ monomers to $A$), and we envisage  either a melt of such diblock copolymers ({\it cf.} Figure \ref{Fig2}, bottom left) or a mixture/blend\footnote{ A similar nonlocal Cahn-Hilliard-like functional models  a blend of diblocks and homopolymers~\cite{CR2}.} of diblocks with homopolymers of type $A$ ({\it cf.} Figure \ref{Fig2}, bottom right). 
\begin{figure} 
\centerline{{\includegraphics[height=0.3in]{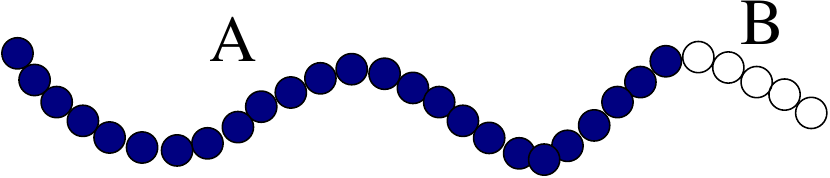}}}
\vspace{1cm}
\centerline{{\includegraphics[height=1.5in]{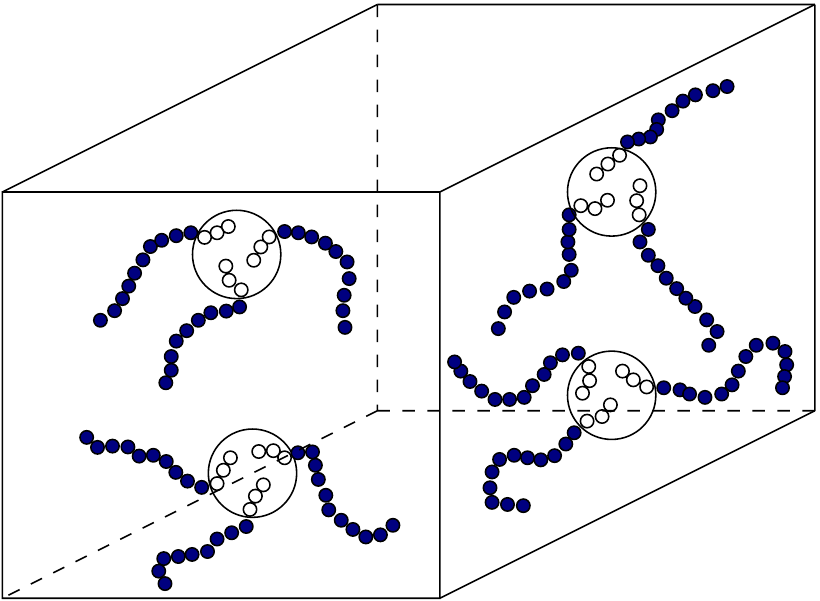}}\qquad \qquad {\includegraphics[height=1.5in]{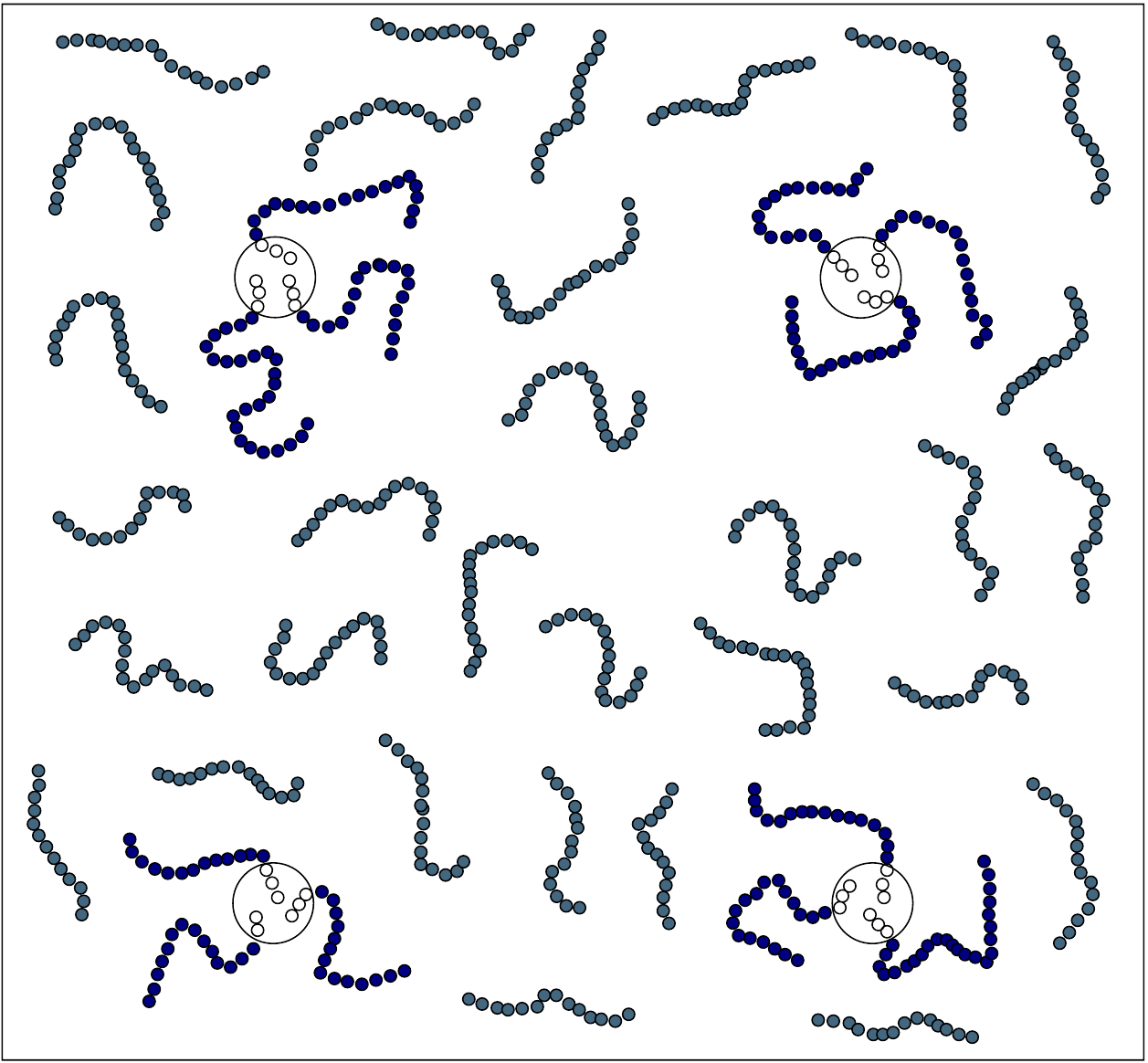}}}
\caption{Top: an AB diblock copolymer macromolecule of minority A composition.  Bottom: 2D schematic of two possible physical scenarios for the regime considered in this article. Left: microphase separation of {\it very long} diblock copolymers with minority A composition. Right: phase separation in a mixture/blend of diblock copolymers and homopolymers of another monomer species having relatively weak interactions with the A and B monomers.}
\label{Fig2}
\end{figure}

This regime is captured by the introduction of a small  parameter $\eta$ and the  appropriate rescaling of the free energy. To this end, we fix a mass parameter $M$ reflecting the total amount of minority phase mass in the limit of delta measures. 
We introduce a small coefficient to $M$, and consider 
phase distributions $u$ such that 
\begin{equation}\label{mass-u}
\int_{\T^n} u \, = \, \eta^n \, M, 
\end{equation}
where $n$ is either $2$ or $3$. 
We rescale $u$ as follows: 
\begin{equation}\label{mass-v}
 v \, : =\, \frac{u}{\eta^n}, 
\end{equation}
so that the new preferred values of $v$ are $0$ and ${1}/{\eta^n}$.  
We now write our free energy (either (\ref{def:Ees}) or (\ref{def:Fes})) in terms of $v$ and rescale in $\eta$ so that the minimum of the free energy remains $O(1)$ as $\eta \rightarrow 0$. 
In this article, we focus our attention on the sharp interface functional  (\ref{def:Fes}): that is,  we 
assume that we have already passed to the limit as $\e \rightarrow 0$, and therefore 
consider the small-volume-fraction  asymptotics of  (\ref{def:Fes}).  
In \cite{CP2} we will show how to extend the results of this paper to the diffuse-interface functional~\pref{def:Ees}, via a diagonal argument with a suitable slaving of $\e$ to $\eta$. 

In Section \ref{sec:degeneration}, we consider a  collection of small particles, 
determine the scaling of the $H^{-1}$-norm, and choose an appropriate scaling of $\gamma$ in terms of  $\eta$ so as to capture a nontrivial limit as $\eta$ tends to $0$. This analysis yields  
\[ {\cal E}(u) \, = \, \begin{cases} 
 \eta \,  \Fea(v) 
  &\text{if } n=2 \\ \\
\eta^2 \, \Feb (v)   &\text{if } n=3,
\end{cases} \]
where 
\begin{equation} \label{Ref-1}
\Fea(v)\, := \,   \eta \int_{\T^2} |\nabla v| 
+\invlogeta\Hmospacea v{\T^2} ^2 \quad {\rm defined \,\, for } \,\,\, v \in BV\left(\T^2; \{0, 1/\eta^2 \}\right)  
\end{equation}
and 
\begin{equation} \label{Ref-2}
\Feb(v)\, := \,  \eta \int_{\T^3} |\nabla v| 
+\eta \Hmospaceb v{\T^3} ^2 \quad {\rm defined \,\, for }  \,\,\, v \in BV\left(\T^3; \{0, 1/\eta^3 \}\right).   
\end{equation}
In both cases, $\Fea(v), \Feb(v)$ remain $O(1)$ as $\eta \rightarrow 0$. 

The aim of this paper is to describe the behavior of these two energies in the limit $\eta\to0$. This will be done in terms of a \emph{$\Gamma$-asymptotic expansion}~\cite{AB93} for  $\Fea(v)$ and $\Feb(v)$. 
That is, we characterize the first and second term in the expansion of, for example,  $\Feb$ of the form 
\[ 
\Feb \, = \, \Fzb \, + \, \eta \, \Hzb \, + \, \hbox{\rm higher order terms}.
\]

Our main results characterize these first- and second-order functionals  $\Fza, \Hza$ (respectively $\Fzb, \Hzb$)  
and  show that:  
\begin{itemize}
\item
At the highest level, the {\it effective energy} is entirely local, i.e.,  the energy {\it  focuses  
separately} on the energy of each particle, and is blind to the spatial distribution of the particles. 
The effective energy contains information about the local structure of the small particles. 
This is presented in three and two dimensions by Theorems \ref{3D-first-order -limit} and \ref{first-order -limit} respectively.

\item At the next level, we see a Coulomb-like 
interaction between the particles. It is this latter part of the energy which we expect 
enforces a periodic array of particles.\footnote{Proving this is a non-trivial matter; see Section~\ref{sec:discussion}} This is presented in three and two dimensions by 
Theorems~\ref{3D-th:sharpnextlevel} and \ref{th:sharpnextlevel} respectively.
\end{itemize}

\bigskip

The paper is organized as follows. Section \ref{sec:notation} contains some basic definitions. 
In  Section  \ref{sec:degeneration} we 
introduce the small parameter $\eta$, and begin with an analysis of the small-$\eta$ behavior of the $H^{-1}$ norm via the 
basic properties of the fundamental solution of the Laplacian in  three and two dimensions. 
We then determine the correct rescalings in dimensions two and three, and arrive at (\ref{Ref-1}) and (\ref{Ref-2}). 
In Section \ref{sec-3dresults} we state the $\Gamma$-convergence results in three dimensions,  together with some properties of the $\Gamma$-limits. The proofs of the three-dimensional results are given in Section~\ref{sec-3dproofs}. In Section~\ref{sec:2d} we state the analogous results in two dimensions and describe the modifications in the proofs. We conclude the paper with a discussion of our results in Section~\ref{sec:discussion}.

\section{Some definitions and notation}
\label{sec:notation}

Throughout this article, we use $\T^n=\R^n/\Z^n$ to denote the $n$-dimensional flat torus of unit volume. For the use of convolution we note that $\T^n$ is an additive group, with neutral element $0\in\T^n$ (the `origin' of $\T^n$).
  For $v \in BV (\T^n; \{0,1\})$ we denote by 
\[ \int_{\T^n} |\nabla v|\]
the total variation measure evaluated on $\T^n$, i.e. $\|\nabla u\| (\T^n)$ \cite{AFP}. 
Since $v$ is the characteristic function of some set $A$, it is simply 
a notion of its perimeter.  
Let $X$ denote the space of Radon measures on $\T^n$. For $ \mu_\eta, \mu \in X$, 
$ \mu_\eta \weak \mu$ denotes weak-$\ast$ measure convergence, i.e.
\[ 
\int_{\T^n} f \, d \mu_\eta \, \, \rightarrow \, \,   \int_{\T^n} f \, d \mu
\]
for all  $f \in C(\T^n)$.  We use the same notation for functions, i.e. 
when writing  $v_\eta\weak v_0$, we interpret $v_\eta$ and $v_0$ as measures whenever necessary.

We  introduce the Green's function  $G_{\T^n}$ for $-\Delta$ in dimension $n$ on $\T^n$. It is the solution of 
\[
- \Delta G_{\T^n} = \delta \, - \, 1,
\qquad\text{with}\qquad
\int_{\T^n} G_{\T^n} = 0,
\]
where $\delta$ is the Dirac delta function at the origin. 
In two dimensions, the Green's function $G_{\T^2}$ satisfies
\begin{equation}
\label{eq:G_T-g2}
G_{\T^2} (x) \, = \, -\frac1{2\pi} \log |x| \, + \, { g}^{(2)}(x)
\end{equation}
for all $x=(x_1,x_2)\in \R^2 $ with $\max\{|x_1|,|x_2|\}\leq 1/2$, where the function $g^{(2)}$ is continuous on $[-1/2,1/2]^2$ and  $C^\infty$ in a neighborhood of 
the origin.
In three dimensions, we have 
 \begin{equation}
\label{eq:G_T-g3}
G_{\T^3} (x) \, = \, \frac1{4 \pi |x|}\,  +\,  g^{(3)}(x)
\end{equation}
for all $x=(x_1,x_2, x_3)\in \R^3 $ with $\max\{|x_1|,|x_2|, |x_3|\}\leq 1/2$, where the function $g^{(3)}$  is again continuous on $[-1/2,1/2]^3$ and smooth in a neighbourhood of the origin.

For $\mu \in X$ such that $\mu (\T^n) = 0$, we may solve 
\[ 
-\Delta v \, = \, \mu,  
\]
in the sense of distributions on $\T^n$.
If $v \in H^1 (\T^n)$, then  $\mu \in H^{-1} (\T^n)$, and 
 \[ \|\mu\|_{H^{-1}(\T^n)}^2 \, := \, \int_{\T^n} |\nabla v|^2 \, dx.  \]
In particular, if $u \in L^2({\T^n})$ then $\left(u - \dashint u\right) \in H^{-1} (\T^n)$ and   
\[\Hmospace{u}{\T^n}^2 
 \, = \, \int_{\T^n} \int_{\T^n} u(x)u(y) \, G_{\T^n} (x-y)\, dx\, dy.
\]
Note that on the right-hand side we may write the function $u$ rather than its zero-average version $u-\dashint u$, since the function $G_{\T^n}$ itself is chosen to have zero average.

We will also need an expression for the $H^{-1}$ norm of the characteristic function of a set of finite perimeter on all of $\R^3$. To this end, let $f$ be such a function and  define 
\[ \|f\|_{H^{-1}({\R^3})}^2 \, = \, \int_{\R^3} |\nabla v|^2\, dx, \]
where $ -\Delta v = f$ on ${\R^3}$ with  $|v| \rightarrow 0$ as $|x| \rightarrow \infty$. 

\section{The small parameter $\eta$, degeneration of the  $H^{-1}$-norm, and the rescaling of (\ref{def:Fes})} 
\label{sec:degeneration}

We  introduce a new parameter $\eta$ controlling the vanishing volume. That is, we consider the total mass to 
be $\eta^n M$, for some fixed $M$, and rescale 
as \[ v_\eta = \frac{u}{\eta^n}.\]
This will facilitate the convergence to Dirac delta measures  of total mass  $M$ and will lead to functionals defined over functions 
$v_\eta: {\T^n} \rightarrow \{0, 1/\eta^n\}$. Note that this transforms the characteristic function $u$ of mass  $\eta^n M$ 
to a function $v_\eta$ with mass $M$,  i.e., 
\[ \int_{\T^n} u=\eta^n M \qquad {\rm while} \qquad 
 \int_{\T^n} v_\eta= M. \] 

On the other hand, throughout our analysis with functions taking on two values $\{0, 1/\eta^n\}$,  we 
 will often need to rescale back to characteristic functions in a way such that {\it the mass is conserved}. To this end, let us fix some notation which we will use throughout the sequel. Consider a  collection $v_\eta: {\T^n} \rightarrow \{0, 1/\eta^n\}$ of  components of the form 
\begin{equation}\label{formv}
v_\eta \, = \, \sum_i v_\eta^i, \qquad v_\eta^i \, = \, 
\frac{1}{\eta^n} \chi^{}_{A_i},
\end{equation}
 where the $A_i$ are disjoint, connected subsets of $\T^n$. 
Moreover, we will always be able to assume\footnote{We will show in the course of the proofs that this basic \emph{Ansatz} of separated connected sets  of small diameter  is in fact generic 
for a sequence of bounded mass and energy 
 ({\it cf.} Lemma \ref{lemma:other_sequence}).} without loss of generality that the 
$A_i$ have a diameter\footnote{For the definition of {\it diameter}, we first note that the torus $\T^n$ has an induced metric
\[
d(x,y) := \min\{|x-y-k|: k\in \Z^n\}\qquad\text{for }x,y\in \T^n.
\] The diameter of a set is then defined in the usual way,
\[
\diam A := \sup \{d(x,y): x,y\in A\}.
\]} less than $1/2$. 
Thus by  associating the torus $\T^n$ with 
$[-1/2,1/2]^n$, we may assume that the $A_i$ do not intersect the boundary $\partial [-1/2,1/2]^n$ and hence we may 
trivially extend $v_\eta^i $ to $\R^n$ by defining it to be zero for $x\not\in A_i$. 
In this extension the total variation of  $v_\eta^i $ calculated on the torus is preserved when calculated over all of $\R^n$. 
We may then transform the components $v_\eta^i$, to functions $z_\eta^i:\R^n\to\R$ by a mass-conservative rescaling that maps their amplitude to 1, i.e., set
\begin{equation}\label{def-z}
z_\eta^i(x) := \eta^n v_\eta^i(\eta x).
\end{equation}

\bigskip

We first consider the case $n=3$. 
Consider a sequence  of functions  $v_\eta$ of the form~\pref{formv}.  
The norm $\Hmo {v_\eta}^2$ can be split up as
\begin{align}
\label{eq:rewrite-Hmo1}
\Hmospace {v_\eta}{\T^3}^2 &\; =\;  \sum_{i=1}^\infty \int_{\T^3}\int_{\T^3} v_\eta^i(x)v_\eta^i(y) \, G_{\T^3}(x-y)\, dxdy \nonumber \\
&  \qquad + \,\, \sum_{\substack{i,j=1\\i\not=j}}^{\infty}
\int_{\T^3}\int_{\T^3} v_\eta^i(x)v_\eta^j(y) \, G_{\T^3}(x-y)\, dxdy.
\end{align}
As we shall see ({\it cf.} the proof of Theorem~\ref{3D-first-order -limit}), 
in the limit $\eta\to0$ it is the  first sum, containing the diagonal terms, that dominates.
For these terms we have 
\begin{align}\label{3D-decomp}
\hbox{}&\hspace*{-5mm}\Hmospace{v_\eta^i}{\Tb}^2 
=\int_\Tb \int_\Tb v_\eta^i(x)v_\eta^i(y) \, G_\Tb (x-y)\, dxdy \nonumber \\
&=\int_{\Tb} \int_{\Tb} v_\eta^i(x)v_\eta^i(y)\,\frac1{4\pi} |x-y|^{-1}\, dxdy
  {}+ \int_\Tb \int_\Tb v_\eta^i(x)v_\eta^i(y) \, g^{(3)}(x-y)\, dxdy\nonumber\\
&=\eta^{-6}\int_{\R^3} \int_{\R^3} z_\eta^i(x/\eta)z_\eta^i(y/\eta)
  \,\frac1{4\pi} |x-y|^{-1}\, dxdy+{}\nonumber\\
&\hspace{1cm} 
  {}+ \int_\Tb \int_\Tb v_\eta^i(x)v_\eta^i(y)\,  g^{(3)}(x-y)\, dxdy\nonumber\\
&=\eta^{-1} \int_{\R^3}\int_{\R^3} 
    z_\eta^i(\xi)z_\eta^i(\zeta) \,\frac1{4\pi}|\xi-\zeta|^{-1}\, d\xi d\zeta
  {}+ \int_\Tb \int_\Tb v_\eta^i(x)v_\eta^i(y) \, g^{(3)}(x-y)\, dxdy\nonumber\\
&= \eta^{-1}\|z_\eta^i\|_{H^{-1}(\R^3)}^2 
  + \int_\Tb \int_\Tb v_\eta^i(x)v_\eta^i(y) \, g^{(3)}(x-y)\, dxdy.
\end{align}
This calculation shows that if the transformed components $z_\eta^i$ converge in a `reasonable' sense, then the dominant behavior of the $H^{-1}$-norm of the original sequence $v$ is given by the term
\[
\frac1\eta \sum_i \|z_\eta^i\|_{H^{-1}(\R^3)}^2 = O\Bigl(\frac1\eta\Bigr).
\]
This argument shows how in the leading-order term only information about the local behavior of each of the separate components enters. The position information is lost, at this level; we will recover this in the study of the next level of approximation.

Turning to the energy, we calculate
\begin{align}
{\cal E} (u) & \;=\;    \int_{\T^3} \,  |\nabla u| 
 \, \, + \,\, \gamma
 \, \Hmospace{u}{\T^3}^2 \notag\\
 & \;=\;  {\eta^3} \,   \int_{\T^3} \,  |\nabla v|   \, \, + \,\, {\gamma}\, {\eta^6} 
  \, \Hmospace{v}{\T^3}^2 \notag\\
&\; =\;  {\eta^2} \left( {\eta} \,  \int_{\T^3} \,  |\nabla v|   \, \, + \,\, \gamma\, \eta^4 
  \, \Hmospace{v}{\T^3}^2 \right). 
\label{rescaling-E-3d}
\end{align}
Note that if $v_\eta$ consists of $N=O(1)$ particles of typical size $O(\eta)$, then 
\[
 {\eta} \,  \int_{\T^3} \,  |\nabla v|   \, \sim \, O(1).
\]
Prompted by (\ref{3D-decomp}),  we expect to make both terms in~\pref{rescaling-E-3d} of the same order by setting
\[
\gamma = \frac{1}{\eta^3}.
\]
Therefore we define
\[
\Feb(v) := \frac{1}{\eta^{2}}{\cal E} (u)  \, = \, \begin{cases}
  \eta\,  \int_\Tb |\nabla v| 
+\eta \,  \Hmospace v\Tb ^2 
  &\text{if } v\in BV(\Tb; \{0,1/\eta^3\})\\
  \infty &\text{otherwise}.
\end{cases}
\]

\medskip

We now switch to the case $n = 2$.  
Here the critical scaling of the $H^{-1}$ in two dimensions causes a different behavior:\begin{align}
\qquad&\hskip-2em\int_\Ta\int_\Ta v_\eta^i(x)v_\eta^i(y) \, G_\Ta (x-y)\, dxdy = \notag\\
 &=-\frac1{2\pi}\int_\Ta\int_\Ta v_\eta^i(x)v_\eta^i(y) \log|x-y|\, dxdy+  \int_\Ta\int_\Ta v_\eta^i(x)v_\eta^i(y) \, g^{(2)}(x-y)\, dxdy\notag\\
 &=-\frac1{2\pi}\int_{\R^2}\int_{\R^2} z_\eta^i(x)z_\eta^i(y) \log\bigl|\eta(x-y)\bigr|\, dxdy
   +  \int_\Ta\int_\Ta v_\eta^i(x)v_\eta^i(y) \, g^{(2)}(x-y)\, dxdy\notag\\
 &=-\frac1{2\pi} \left(\int_{\R^2} z_\eta^i\right)^2 \log\eta-\frac1{2\pi} \int_{\R^2}\int_{\R^2} z_\eta^i(x)z_\eta^i(y) \log |x-y|\, dxdy \notag\\
 &\qquad \qquad+  \int_\Ta\int_\Ta v_\eta^i(x)v_\eta^i(y) \, g^{(2)} (x-y)\, dxdy\notag\\
 &=\frac1{2\pi} \left(\int_{\R^2} z_\eta^i\right)^2 \logeta
    -\frac1{2\pi} \int_{\R^2}\int_{\R^2} z_\eta^i(x)z_\eta^i(y) \log |x-y|\, dxdy \notag\\
 &\qquad\qquad +  \int_\Ta\int_\Ta v_\eta^i(x)v_\eta^i(y) \, g^{(2)}(x-y)\, dxdy .
 \label{eq:rewrite-Hmo2}
\end{align}
By this calculation we expect that the dominant behavior of the $H^{-1}$-norm of the original sequence $v$ is given by the term
\begin{equation}
\label{eq:local_L1Hmo}
\sum_i \frac1{2\pi} \left(\int_{\R^2} z_\eta^i\right)^2 \logeta
\, =\,  \frac\logeta{2\pi} \sum_i \left(\int_{\Ta} v_\eta^i\right)^2.
\end{equation}
Note how, in contrast to the three-dimensional case, only the distribution of the mass of $v$ over the different components enters in the limit behavior. 
Note also that the critical scaling here is $\logeta$. 

Following the same line as for the three-dimensional case, and setting
\begin{equation}
\label{def:u-v-scaling}
v = \frac u{\eta^2},
\end{equation} 
we calculate
\begin{eqnarray*}
{\cal E} (u) & = &    \int_{\T^2} \,  |\nabla u| 
 \, \, + \,\, \gamma
 \, \Hmospace{u}{\T^2}^2 \\
 & = & {\eta^2} \,   \int_{\T^2} \,  |\nabla v|   \, \, + \,\, {\gamma}\, {\eta^4} 
  \, \Hmospace{v}{\T^2}^2 \\
& = & {\eta} \left( {\eta} \,  \int_{\T^2} \,  |\nabla v|   \, \, + \,\, \gamma\, \eta^3 
  \, \Hmospace{v}{\T^2}^2 \right). 
\end{eqnarray*}
Following (\ref{eq:rewrite-Hmo2}), (\ref{eq:local_L1Hmo}), 
in order to capture a nontrivial limit we must choose 
\[ \gamma \, = \, \frac{1}{\logeta\, \eta^{3}}. \]
With this choice of $\gamma$, we define 
\[
\Fea(v) := \frac1{\eta}{\cal E} (u)  \, = \, \begin{cases}
 \,  \eta\int_\Ta |\nabla v| 
+\invlogeta\Hmospace v\Ta ^2 
  &\text{if } v\in BV(\Ta; \{0,1/\eta^2\})\\
  \infty &\text{otherwise}. 
\end{cases}
\]


\section{Statement of the main results in three dimensions}
\label{sec-3dresults}

We now state precisely the $\Gamma$-convergence results for $\Feb$ in three dimensions. 
Both our $\Gamma$-limits will be defined over countable sums of weighted Dirac delta measures  $\sum_{i=1}^\infty m^i \delta_{x^i}$.  
We start with the first-order limit. To this end, let us introduce 
the function 
\begin{align}
\fzb(m) &:= \inf\left\{\int_{\R^3} |\nabla z| + \|z\|_{H^{-1}(\R^3)}^2: 
   z\in BV(\R^3;\{0,1\}), \int_{\R^3} z = m \right\}. 
\label{def:fe3}
\end{align}
 
We also define the limit functional\footnote{The definition of~$\Fzb$ requires the point mass positions $x^i$ to be distinct, and the reader might wonder why this is necessary. Consider the following functional, which might be seen as an alternative,
\[
\widetilde \Fzb(v) := \begin{cases}
\sum_{i=1}^\infty \fzb(m^i) & \text{if } v = \sum_{i=1}^\infty m^i \delta_{x^i} \text{ with } m^i\geq 0,\\
\infty & \text{otherwise}.
\end{cases}
\]
This functional is actually not well defined: the function $v$ will have many representations (of the type $\delta = a\delta + (1-a)\delta$, for any $a\in(0,1)$) that will not give rise to the same value of the functional. Therefore the functional $\widetilde \Fzb$ is a functional of the representation, not of the limit measure $v$. The restriction to distinct $x^i$ eliminates this dependence on representation.}
\begin{align*}
\Fzb(v) := \begin{cases}
\sum_{i=1}^\infty \fzb(m^i) & \text{if } v = \sum_{i=1}^\infty m^i \delta_{x^i}, \,\,\,  \{x^i\} \text{ distinct, and }m^i\geq0\\
\infty & \text{otherwise}.
\end{cases}
\end{align*}
\begin{remark}\label{remark-lsc}
Under weak convergence multiple point masses may join to form a single point mass. The functional $\Fzb$ is lower-semicontinuous under such a change if and only if the function $\fzb$ satisfies the related inequality
\begin{equation}\label{e-0-subadd}
 \fzb\Bigl(\sum_{i=1}^\infty m^i\Bigr) \, \le \, \sum_{i=1}^\infty \fzb(m^i). 
 \end{equation}
The function $\fzb$ does satisfy this property, as can be recognized by taking approximating functions $z^i$ and translating them far from each other; the sum $\sum_i z^i$ is admissible and its limiting energy, in the limit of large separation, is the sum  of the individual energies.
\end{remark}
 Having  introduced the limit functional $\Fzb$, we are now in a position to state the first  main result of this paper. 

\begin{theorem}
\label{3D-first-order -limit}
Within the space $X$, we have 
\[
\Feb \, \stackrel{\Gamma}\longrightarrow \, \Fzb \qquad {\rm as} \quad \eta \rightarrow 0.
\]
That is, 
 
\begin{itemize} 
\item (Condition 1 -- the lower bound and compactness) 
Let $v_\eta$ be a sequence such that the sequence of energies
$\Feb(v_\eta)$ is bounded. Then
(up to a subsequence) $v_\eta\weak v_0$, $\supp v_0$ is countable, and 
\begin{equation}
\label{lb:sharp-interface}
\liminf_{\eta\to0} \Feb(v_\eta) \geq \Fzb(v_0).
\end{equation}
\item (Condition 2 -- the upper bound) 
Let $\Fzb(v_0)<\infty$. 
Then there exists a sequence $v_\eta \weak v_0$ such that 
\[
\limsup_{\eta\to0} \Feb(v_\eta) \leq \Fzb(v_0).
\]
\end{itemize} 
\end{theorem}

Note that the compactness condition which usually accompanies a Gamma-convergence result has been built into Condition 1 (the lower bound). 
The fact that sequences with bounded energy $\Feb$ converge to a collection of delta functions is partly so by construction: the functions $v_\eta$ are positive, have uniformly bounded mass, and only take values either $0$ or $1/\eta^3$. Since $\eta\to0$, the size of the support of $v_\eta$ shrinks to zero, and along a subsequence $v_\eta$ converges in the sense of measures to a limit measure; in line with the discussion above, this limit measure is shown to be a sum of Dirac delta measures (Lemma~\ref{lemma:compactness}). 

We have the following properties of $\fzb$, and a characterization of minimizers of $\Fzb$. The proof is presented in Section \ref{sec:proofs2}. 
\begin{lemma}
\label{lemma:char_limit}
\begin{enumerate}
\item For every $a>0$, $\fzb'$ is non-negative and bounded from above on $[a,\infty)$. 
\item $\fzb$ is strictly concave on $[0,2\pi ]$. 
\item 
If $\{m^i\}_{i\in\N}$ with $\sum_i m^i<\infty$ satisfies
\begin{equation}
\label{cond:mass-minimizer}
\sum_{i=1}^\infty \fzb(m^i) = \fzb\Bigl(\sum_{i=1}^\infty m^i\Bigr),
\end{equation}
then
 only a finite number of $m^i$ are non-zero.
\end{enumerate}
\end{lemma}

Note that  the limit functional $\Fzb$ is blind to positional information: the value of $\Fzb$ is independent of the positions $x^i$ of the point masses. In order to capture this positional information, we consider the next level of approximation, by subtracting the minimum of $\Fzb$ and renormalizing the result. To this end, note that  
among all measures of mass $M$, the global minimizer of $\Fzb$ is given by
\[
\min \left\{\Fzb(v): \int_{\Tb} v = M\right\} = \fzb(M).
\]
We recover the next term in the expansion as the limit of $\Feb - \fzb$, appropriately rescaled, that is of the functional
\[ 
\Heb(v_\eta) := \eta^{-1}\left[ \Feb(v_\eta) - \fzb\left(\int_\Tb v_\eta\right)\right] .
\]
If this second-order energy remains bounded in the limit $\eta\to0$, then the limiting object $v_0 = \sum_i m^i \delta_{x^i}$ necessarily has two properties:
\begin{enumerate}
\item The limiting mass weights $\{ m^i\}$ satisfy~\pref{cond:mass-minimizer};
\item For each $m^i$, the minimization problem defining $\fzb(m^i)$ has a minimizer. 
\end{enumerate}
The first property above arises from the condition that $\Feb(v_\eta)$ converges to its minimal value as $\eta\to0$. The second is slightly more subtle, and can be understood by the following formal scaling argument. 

In the course of the proof we construct truncated versions of $v_\eta$, called $v_\eta^i$, each of which is localized around the corresponding limiting point $x^i$ and rescaled as in~\pref{def-z} to a function $z_\eta^i$. For each $i$ the sequence $z_\eta^i$ is a minimizing sequence for the  minimization problem $\fzb(m^i)$, and the scaling of $\Heb$ implies that the energy $\Feb(v_\eta)$ converges to the limiting value at a rate of at least $O(\eta)$. In addition, since $v_\eta^i$ converges to a delta function, the typical spatial extent of $\supp v_\eta^i$ is of order $o(1)$, and therefore the spatial extent of $\supp z_\eta^i$ is of order $o(1/\eta)$. If the sequence $z_\eta^i$ does not converge, however, then it splits up into separate parts; the interaction between these parts is penalized by the $H^{-1}$-norm at the rate of $1/d$, where $d$ is the distance between the separating parts. Since $d=o(1/\eta)$, the energy penalty associated with separation scales larger than $O(\eta)$, which contradicts the convergence rate mentioned above. 

This is no coincidence; the scaling of $\Heb$ has been chosen just so that the interaction between objects that are separated by $O(1)$-distances in the original variable $x$ contributes an $O(1)$ amount to this second-level energy. If they are asymptotically closer, then the interaction blows up. 

Motivated by these remarks we define the set of admissible limit sequences
\[
{\cal M}
:= \left\{\{m^i\}_{i\in\N}: m^i\geq 0, \ \text{satisfying \pref{cond:mass-minimizer}, such that $\fzb(m^i)$ admits a minimizer for each $i$}\right\}.
\]  
The limiting energy functional $\Hzb$ can already be recognized in the decomposition given by~\pref{eq:rewrite-Hmo1} and \pref{3D-decomp}. 
We show in the proof in Section~\ref{sec-3dproofs} that the interfacial term in the energy $\Feb$ is completely cancelled by the corresponding term in $\fzb$, as is the highest-order term in the expansion of $\Hmo {v_\eta}^2$. What remains is a combination of 
cross terms,
\[
\sum_{\substack{i,j=1\\i\not=j}}^{\infty}
\int_\Tb\int_\Tb v_\eta^i(x)v_\eta^j(y) G_\Tb(x-y)\, dxdy,
\]
and lower-order self-interaction parts of the $H^{-1}$-norm.
\[
\sum_{i=1}^\infty  \int_\Tb\int_\Tb v_\eta^i(x)v_\eta^i(y) g^{(3)} (x-y)\, dxdy.
\]
With these remarks we define
\[
\Hzb (v) := \begin{cases}
\displaystyle
\sum_{i=1}^\infty g^{(3)}(0) \, (m^i)^2 \, +\, & \\
 \qquad \sum_{{i\not=j}} m^im^j\, G_\Tb (x^i-x^j)
  & \displaystyle\text{if } v = \sum_{i=1}^n m^i \delta_{x^i} \text{ with } \{x^i\} \text{ distinct,  } \{m^i\}\in{\cal M}\\
\infty & \text{otherwise}.
\end{cases}
\]

We have:  
\begin{theorem}
\label{3D-th:sharpnextlevel}
Within the space $X$, we have 
\[
\Heb \stackrel{\Gamma}\longrightarrow \Hzb \qquad {\rm as} \quad \eta \rightarrow 0.
\]
That is, Conditions $1$ and $2$ of Theorem~$\ref{3D-first-order -limit}$ hold with $\Feb$ and $\Fzb$ replaced with 
$\Heb$ and $\Hzb$. 
\end{theorem}

\medskip

\noindent The interesting aspects of this limit functional $\Hzb$ are
\begin{itemize}
\item  In contrast to $\Fzb$, the functional $\Hzb$ is only finite on \emph{finite} collections of point masses, which in addition satisfy two constraints: the collection should satisfy~\pref{cond:mass-minimizer}, and each weight $m^i$ should be such that the corresponding minimization problem~\pref{def:fe3} is achieved. In Section~\ref{sec:discussion} we discuss these properties further. 
\item The main component of $\Hzb$ is the two-point interaction energy
\[
\sum_{i,j:\ i\not=j} m^{i}m^{j}G_\Tb(x^i-x^j).
\]
This two-point interaction energy is known as a Coulomb interaction energy, by reference to electrostatics. A similar limit functional also appeared in~\cite{RW5}.
\end{itemize}



\section{Proofs of Theorems \ref{3D-first-order -limit} and \ref{3D-th:sharpnextlevel} } 
\label{sec-3dproofs}

\subsection{Concentration into point measures}
\begin{lemma}[Compactness]
\label{lemma:compactness}
Let $v_\eta$ be a sequence in $BV (\Tb; \{0, 1/\eta^3\})$ such that both $\int_\Tb v_\eta$ and $\Feb(v_\eta)$ are uniformly bounded. Then there exists a subsequence such that $v_\eta\weakto v_0$ as measures, where 
\begin{equation}
\label{eq:structure_v0}
v_0 := \sum_{i=1}^\infty m^i \,  \delta_{x^i},
\end{equation}
with $m^i\geq0$ and $x^i\in \Tb$ distinct. 
\end{lemma}

Note that we often write ``a sequence $v_\eta$" instead of ``a sequence $\eta_n\to0$ and a sequence~$v_n$" whenever this does not lead to confusion. 
The essential tool to prove convergence to delta measures  is 
the {\it Second Concentration Compactness Lemma} of Lions \cite{L1}.

\begin{proof} The functions $w_\eta := \eta v_\eta$ satisfy $w_\eta\to 0$ in $L^1(\Tb)$, and $|\nabla w_\eta|\, =\, \eta\, |\nabla v_\eta|$  bounded in $L^1(\Tb)$. On the other hand, by definition, one has  $w_\eta^{3/2}\, =\, v_\eta$ which is  bounded in $L^1(\Tb)$. Hence we extract a subsequence such that $v_\eta \weakto v_0$ as measures. Lemma~I.1  (i) of~\cite{L1} (with $m=p=1, q=3/2$) then implies that $v_0$ has the structure~\pref{eq:structure_v0}.
\end{proof}

\bigskip

The proof of the two lower-bound inequalities uses a partition of $\supp v_\eta$ into disjoint sets with positive pairwise distance. This division implies the inequality
\[
\int_{\Tb} |\nabla v_\eta| = \sum_i \int_\Tb |\nabla v_\eta^i|,
\]
and is an important step towards the separation of local and global effects in the functionals.
The following lemma provides this partition into disjoint particles. 

\begin{lemma}
\label{lemma:other_sequence}
Continue under the conditions of the previous lemma. For the purpose of proving a lower bound on $\Feb(v_\eta)$ and $\Heb(v_\eta)$ we can assume without loss of generality that for some $n\in \N$
\[
v_\eta = \sum_{i=1}^n v_\eta^i 
\]
with $\wliminf_{\eta\to0} v_\eta^i\geq  m_0^i\delta_{x^i}$ as measures, $\dist(\supp v_\eta^i,\supp v_\eta^j)>0$ for all $i\not=j$, and $\diam\supp v_\eta^i<1/4$. In addition, for the lower bound on $\Heb(v_\eta)$ we can assume that for each $i$, $v_\eta^i\weakto m_0^i \delta_{x^i}$ and that  there exist $\xi_\eta^i\in \T^3$ and a constant $C^i>0$ such that
\begin{equation}
\label{cond:second-moments}
\int_\Tb |x-\xi_\eta^i|^2 v_\eta^i(x)\, dx \leq C^i\eta^2.
\end{equation}
\end{lemma}

The proof of Lemma \ref{lemma:other_sequence} is given in detail in Section \ref{sec:proofs2}. A central ingredient is the following truncation lemma. Here $\Omega$ is either the torus $\T^3$ or an open bounded subset of~$\R^3$.

\begin{lemma}
\label{lemma:truncation}
Let $n\in\N$ be fixed, let $a_k\to\infty$, and let $u_k\in BV(\Omega;\{0,a_k\})$ satisfy
\begin{equation}
\label{ass:small_perimeter}
\int_\Omega |\nabla u_k| = o(a_k),
\end{equation}
and converges weakly in $X$ to a weighted sum
\[
\sum_{i=1}^\infty m^i\delta_{x^i},
\]
where $m^i\geq0$ and the $x^i\in \Omega$ are distinct.  Then there exist components $u_k^i\in BV(\Omega;\{0,a_k\})$, $i=1\dots n$, satisfying $\diam\supp u_k^i\leq 1/4$, $\inf_k \inf_{i\not= j} \dist(\supp u_k^i,\supp u_k^j)>0$, and
\begin{equation}
\label{prop:tilde-u-k}
\wliminf_{k\to\infty} u_k^i\geq m^i\delta_{x^i}, 
\end{equation}
in the sense of distributions. 
In addition, the modified sequence $\tilde u_k = \sum_i u_k^i$ satisfies
\begin{enumerate}
\item \label{enum:lemma-truncation-1}
$\tilde u_k\leq u_k$ for all $k$;
\item \label{enum:lemma-truncation-2}
$\limsup_{k\to\infty}\int (u_k-\tilde u_k)\leq\sum_{i=n+1}^\infty m^i$;
\item There exists a constant $C=C(n)>0$ such that for all $k$
\begin{equation}
\label{ineq:perimeter-quantitative}
\int |\nabla \tilde u_k| \leq \int |\nabla u_k| - C\|u_k-\tilde u_k\|_{L^{3/2}(\Omega)}.
\end{equation}
\end{enumerate}
\end{lemma}
The essential aspects of this lemma are the construction of a new sequence {which again lies in $BV(\Omega;\{0,a_k\})$}, and the quantitative inequality~\pref{ineq:perimeter-quantitative} relating  the perimeters.

\subsection{Proof of Theorem  \ref{3D-first-order -limit}}
\begin{proof} 
 {\bf (Lower bound)} 
Let  $v_\eta$ be a sequence such that the sequences of energies
$\Feb(v_\eta)$ and masses $\int_\Tb v_\eta$ are bounded. By Lemma~\ref{lemma:compactness}, a subsequence converges to a limit $v_0$ of the form~\pref{eq:structure_v0}. By Lemma~\ref{lemma:other_sequence} it is sufficient to consider a sequence (again called $v_\eta$) such that $v_\eta = \sum_{i=1}^n v_\eta^i$
with $\wliminf_{\eta\to0}v_\eta^i\geq m_0^i\delta_{x^i}$, $\supp v_\eta^i\subset B(x^i,1/4)$, and $\dist(\supp v_\eta^i,\supp v_\eta^j)>0$ for all $i\not=j$. Then, writing 
\begin{equation}\label{def-z-eta}
z_\eta^i(y) := \eta^3 v_\eta^i\bigl(x^i+\eta y), 
\end{equation} 
we have
\[
\int_\Tb v_\eta^i = \int_{\R^3} z_\eta^i
\qquad\text{and}\qquad
\int_\Tb |\nabla v_\eta^i| = \eta^{-1} \int_{\R^3}|\nabla z_\eta^i|, 
\]
and by~\pref{3D-decomp}
\[
\Hmospace{v_\eta^i}{\Tb}^2 
= \eta^{-1}\|z_\eta^i\|_{H^{-1}(\R^3)}^2 
  + \int_\Tb \int_\Tb v_\eta^i(x)v_\eta^i(y) g^{(3)}  (x-y)\, dx\, dy.
\]
For future use we introduce the shorthand
\[
m^i_\eta := \int_\Tb v_\eta^i  = \int_{\R^3} z_\eta^i.
\]
Then
\begin{align}
\label{3d-decomposition}
\Feb(v_\eta)&= \sum_{i=1}^n \Feb(v_\eta^i) 
  + \eta\sum_{\substack{i,j=1\\i\not=j}}^n \int_\Tb \int_\Tb v_\eta^i(x)v_\eta^j(y) G_\Tb (x-y)\, dx\, dy \nonumber \\
&=\sum_{i=1}^n \left[ \int_{\R^3} |\nabla z_\eta^i |  \, + \, \|z_\eta^i\|_{H^{-1}(\R^3)}^2\right] \nonumber
\\
&\qquad{}
   + \eta\sum_{i=1}^n \int_\Tb \int_\Tb v_\eta^i(x)v_\eta^i(y) g^{(3)}  (x-y)\, dx\, dy  
   + \eta\sum_{\substack{i,j=1\\i\not=j}}^n \int_\Tb \int_\Tb v_\eta^i(x)v_\eta^j(y) G_\Tb (x-y)\, dx\, dy \nonumber \\
&\geq \sum_{i=1}^n \fzb\left(m_\eta^i\right)
  + \eta\inf g^{(3)} \sum_{i=1}^n \left(m_\eta^i\right)^2
  + \eta\inf G_\Tb\sum_{\substack{i,j=1\\i\not=j}}^n m_\eta^im_\eta^j.
\end{align}
Since the last two terms vanish in the limit,  the continuity and monotonicity of $\fzb$ (a consequence of Lemma \ref{lemma:char_limit}) imply  that 
\[
\liminf_{\eta\to0}\Feb(v_\eta) \geq \sum_{i=1}^n \fzb\left(\liminf_{\eta\to0} m_\eta^i\right)\geq \sum_{i=1}^n \fzb(m^i)\geq 
\Fzb(v_0).
\]

\bigskip

{\bf (Upper bound)}
Let $v_0$ satisfy 
$\Fzb(v_0)<\infty$. It is  sufficient to prove the statement for finite sums
\[
v_0 = \sum_{i=1}^n m^i\, \delta_{x^i}, 
\]
since an infinite sum $v_0=\sum_{i=1}^\infty m^i\, \delta_{x^i}$ can trivially be approximated by finite sums, and in that case
\[
\Fzb\left(\sum_{i=1}^n m^i\, \delta_{x^i}\right)
  = \sum_{i=1}^n \fzb(m^i) 
  \leq \sum_{i=1}^\infty \fzb(m^i) = \Fzb(v_0).
\]

To construct the appropriate sequence $v_\eta \weak v_0$,  let $\epsilon >0$ and 
let $z^i$  be  near-optimal in the definition of $e_0^{\rm 3d} (m^i)$, i.e., 
\begin{equation}\label{optimalz}
 \int_{\R^3} |\nabla z^i| \,\, + \,\,  \|z^i\|_{H^{-1}(\R^3)}^2 \,\, \leq\,\, e_0^{\rm 3d} (m^i)  + \, \, \frac\epsilon n.
\end{equation}
By an argument based on the isoperimetric inequality we can assume that the support of $z^i$ is bounded. We then set
\begin{equation}\label{v-z}
v_\eta^i(x) := \eta^{-3}z^i(\eta^{-1}(x-x^i)),
\end{equation}
so that 
\[
\int_\Tb v_\eta^i =  m^i. 
\]
Since the diameters of the supports of the $v_\eta^i$ tend to zero, and since the $x^i$ are distinct, 
$v_\eta : = \sum_i v_\eta^i $ is admissible for $\Feb$ when $\eta$ is sufficiently small.

Following the  argument of \pref{3d-decomposition}, we have    
\begin{eqnarray*}
\Feb (v_\eta) & = & 
\sum_{i=1}^n \left[ \int_{\R^3} |\nabla z^i |  \, + \, \|z^i\|_{H^{-1}(\R^3)}^2\right] \nonumber
\\
&& {}
   + \eta\sum_{i=1}^n \int_\Tb \int_\Tb v_\eta^i(x)v_\eta^i(y) g^{(3)}  (x-y)\, dx\, dy  
   + \eta\sum_{\substack{i,j=1\\i\not=j}}^n \int_\Tb \int_\Tb v_\eta^i(x)v_\eta^j(y) G_\Tb (x-y)\, dx\, dy
\end{eqnarray*}
and thus  
\[ \limsup_{\eta \rightarrow 0} \Feb (v_\eta) \, \le  \, \Fzb(v_0) \, + \, \epsilon. \]
The result follows by letting $\epsilon $ tend to zero. 
\end{proof}

\subsection{Proof of Theorem  \ref{3D-th:sharpnextlevel} }
\begin{proof}  {\bf (Lower bound)} 
Let $v_\eta=\sum_{i=1}^nv_\eta^i$ be a sequence with bounded energy $\Heb(v_\eta)$ as given by Lemma~\ref{lemma:other_sequence}, converging to a $v_0$ of the form 
\[ 
v_0 = \sum_{i=1}^n m^i \delta_{x^i},
\]
where $m_0^i\geq0$ and the $x^i$ are distinct. Again we use the rescaling~\pref{def-z-eta} and we set
\[
m^i_\eta := \int_\Tb v_\eta^i  = \int_{\R^3} z_\eta^i.
\]

Following the second line of~\pref{3d-decomposition}
we have 
\begin{align}
\Heb(v_\eta) &= \eta^{-1}\left[ \Feb(v_\eta) - \fzb\left(\int_\Tb v_\eta\right)\right]\notag\\
&= \frac1\eta\sum_{i=1}^n\left[\int_{\R^3} |\nabla z^i_\eta| + \| z^i_\eta\|^2_{H^{-1}(\R^3)} - \fzb\left(m_\eta^i\right)\right]
 +\frac1\eta\left[\sum_{i=1}^n\fzb\left(m_\eta^i\right)- 
                       \fzb\left(\sum_{i=1}^n m_\eta^i\right)\right]      \notag\\
&\qquad {}
   + \sum_{i=1}^n \int_\Tb \int_\Tb v_\eta^i(x)\, v_\eta^i(y)\,  g^{(3)} (x-y)\, dx\, dy  
   + \sum_{\substack{i,j=1\\i\not=j}}^n \int_\Tb \int_\Tb v_\eta^i(x)\, v_\eta^j(y) \, G_\Tb (x-y)\, dx\, dy.
\label{alt:Heb}
\end{align}
Since the first two terms are both non-negative, the boundedness of $\Heb(v_\eta)$ and continuity of $\fzb$  imply that 
\[
0\leq \sum_{i=1}^n \fzb(m^i) - \fzb\left(\sum_{i=1}^n m^i\right)
= \lim_{\eta\to0} \left[\sum_{i=1}^n\fzb\left(m_\eta^i\right)- 
                       \fzb\left(\int_\Tb v_\eta\right)\right]
\leq 0,
\]
and therefore the sequence $\{m^i\}$ satisfies~\pref{cond:mass-minimizer}. 

By the condition~\pref{cond:second-moments} the sequence $z_\eta^i$ is tight, and since it is bounded in $BV(\R^3;\{0,1\})$, a subsequence converges in $L^1(\R^3)$ to a limit $z_0^i$ (see for instance Corollary IV.26 of \cite{Bre}). We then have
\begin{align*}
0&\leq \int_{\R^3} |\nabla z_0^i| + \|z_0^i\|^2_{H^{-1}(\R^3)} - \fzb\left(m^i\right)\\
&\leq \liminf_{\eta\to0} \biggl[\int_{\R^3} |\nabla z_\eta^i| + \|z_\eta^i\|^2_{H^{-1}(\R^3)} \biggr]
 - \lim_{\eta\to0} \fzb\left(m_\eta^i\right)
\stackrel{\pref{alt:Heb}}=0,
\end{align*}
which implies that $z_0^i$ is a minimizer for $\fzb(m^i)$.

Finally we conclude that 
\begin{align*}
\liminf_{\eta\to0} \Heb(v_\eta) 
&\geq 
\liminf_{\eta\to0} \Biggl( 
  \sum_{i=1}^n \int_\Tb \int_\Tb v_\eta^i(x)\, v_\eta^i(y) \, g^{(3)}  (x-y)\, dxdy  \\
 & \qquad \qquad \quad  + \,\,  \sum_{\substack{i,j=1\\i\not=j}}^n \int_\Tb \int_\Tb v_\eta^i(x)\, v_\eta^j(y) \, G_\Tb (x-y)\, dx\, dy \Biggr)\\
& = \,  g^{(3)} (0)\sum_{i=1}^n (m^i)^2
  + \sum_{\substack{i,j=1\\i\not=j}}^n m^i\, m^j \, G_\Tb (x^i-x^j)
  = \Hzb(v_0).
\end{align*}

\bigskip

{\bf (Upper bound)}  Let 
\[
v_0 = \sum_{i=1}^n m^i \delta_{x^i},
\]
with the $x^i$ distinct and $\{m^i\} \in{\cal M}$. By the definition of ${\cal M}$ we may choose $z^i$ that achieve the minimum in the minimization problem defining $\fzb(m^i)$; by an argument based on the isoperimetric inequality the support of $z^i$ is bounded. 

Setting $v_\eta^i$ by~\pref{v-z}, for $\eta$ sufficiently small the function
$v_\eta : = \sum_{i=1}^n v_\eta^i $ is admissible for $\Heb$, and $v_\eta \weak v_0$. 
Then following the second line of (\ref{3d-decomposition}), we have 
\[ \lim_{\eta \rightarrow 0} \Heb (v_\eta) \, = \, \Hzb(v_0). \]

\end{proof}

\subsection{Proofs of Lemmas~\ref{lemma:other_sequence} and~\ref{lemma:truncation}}
\label{sec:proofs2}

For the proof of Lemma~\ref{lemma:other_sequence} we first state and prove two lemmas. 
Throughout this section, if $B$ is a ball in $\R^3$ and $\lambda>0$, then $\lambda B$ is the ball in $\R^3$ obtained by multiplying $B$ by $\lambda$ with respect to the center of $B$; $B$ and $\lambda B$ therefore have the same center.

\begin{lemma}\label{lemma-0}
Let $w\in BV(B_R;\{0,1\})$. Choose $0<r<R$, and set $A := B_R\setminus \overline B_r$. Then for any $r\leq\rho\leq R$ we have
\[
\frac{\H^2(\partial B_\rho \cap \supp w) }{\H^2(\partial B_\rho)}
\leq 
\frac1{\H^2(\partial B_r)}\int_A |\nabla w| + \dashint_{A} w
\]
\end{lemma}

\begin{proof}
Let $P$ be the projection of $\R^3$ onto $\overline B_r$. For any closed set $D\subset \R^3$ with finite perimeter, the projected set $P(A\cap D)$ is included in $E_b\cup E_r$, where the two sets are:
\begin{itemize}
\item The projected boundary $E_b := P(A\cap \partial D)$; since $P$ is a contraction, $\H^2(E_b)\leq \H^2(A\cap \partial D)$;
\item The set of projections of full radii $E_r := \{x\in \partial B_r: \lambda x\in D \text{ for all }1\leq \lambda \leq R/r\}$, for which 
\[
\H^2(E_r) = \frac{\H^2(\partial B_r)}{\L^3(A)}{\L^3(\{\lambda x: x\in E_r, 1\leq \lambda\leq R/r\})}
\leq \frac{\H^2(\partial B_r)}{\L^3(A)}{\L^3(D\cap A)}.
\]
\end{itemize}
Applying these estimates to $D=\supp w$ we find
\begin{align*}
\frac{\H^2(\partial B_\rho \cap \supp w) }{\H^2(\partial B_\rho)}
&= 
\frac{\H^2\bigl(P(\partial B_\rho \cap \supp w)\bigr) }{\H^2(\partial B_r)}\\
&\leq 
\frac{\H^2\bigl(P( A\cap \supp w)\bigr) }{\H^2(\partial B_r)}\\
&\leq 
\frac1{\H^2(\partial B_r)}\left\{\H^2(A\cap \partial \supp w) + \frac{\H^2(\partial B_r)}{\L^3(A)}{\L^3(A\cap\supp w)}\right\},
\end{align*}
and this last expression implies the assertion.
\end{proof}

\begin{lemma}
\label{lemma:alpha}
There exists $0<\alpha<1$ with the following property. For any $w\in BV(B_R;\{0,1\})$ with
\begin{equation}
\label{assumption:RII}
\frac1{\H^2(\partial B_{\alpha R})} \int_{B_R\setminus \overline{B_{\alpha R}}} |\nabla w|
  \;+\; \dashint_{B_R\setminus \overline{B_{\alpha R}}} w 
  \leq \frac12,
\end{equation}
there exists $\alpha\leq \beta < 1$ such that 
\begin{equation}
\label{ineq:cutoff_r}
2\|\partial B_{\beta R}\|(\supp w) \leq \int_{B_R\setminus \overline{B_{\beta R}}}|\nabla w|.
\end{equation}
\end{lemma}

\begin{proof}
By approximating (see for example Theorem 3.42 of \cite{AFP}) and scaling we can assume that $w$ has smooth support and that $R=1$. Set $0<\alpha<1$ to be such that 
\begin{equation}
\label{def:alpha}
\frac{(1-\alpha)^2}{16C} = \H^2(\partial B_\alpha),
\end{equation}
where $C$ is the constant in the relative isoperimetric inequality on the sphere $S^2$:
\[
\min\{\H^2(D\cap S^2),\H^2(S^2\setminus D)\}
\leq C(\H^1(\partial D\cap S^2))^2.
\]
We note that  the combination of the assumption~\pref{assumption:RII} and Lemma~\ref{lemma-0} implies that when applying this inequality to $D= \supp w$, with $S^2$ replaced by $\partial B_{1-s}$, the minimum is attained by the first argument, i.e. we have 
\[
\H^2(D\cap\partial B_{1-s})
 \leq C (\H^1(\partial D\cap\partial B_{1-s}))^2.
\]

We now assume that the assertion of the Lemma is false, i.e. that for all $\alpha<r<1$
\begin{equation}
\label{assumption:perimeter}
0< 2\|\partial B_r\|(D) - \|\partial D\|(B_1\setminus \overline B_r).
\end{equation}
Setting $f(s) := \H^1(\partial D\cap\partial B_{1-s})$ we have 
\begin{equation}
\label{ineq:perimeter2}
\int_0^s f(\sigma)\, d\sigma = \int_{1-s}^1 \H^1(\partial D\cap\partial B_r)\, dr
\leq \int _{B_1\setminus \overline B_{1-s}} |\nabla w|
\stackrel{\pref{assumption:perimeter}}
< 2\|\partial B_{1-s}\|(D).
\end{equation}
By the relative isoperimetric inequality we find
\[
\int_0^s f(\sigma)\, d\sigma < 2\|\partial B_{1-s}\|(D)
 \leq 2C (\H^1(\partial D\cap\partial B_{1-s}))^2
= 2Cf(s)^2.
\]
Note that this inequality implies that $f$ is strictly positive for all $s$. Solving this inequality for positive functions~$f$ we find
\[
\int_0^{1-\alpha} f(\sigma)\, d\sigma > \frac{(1-\alpha)^2}{8C}
\stackrel{\pref{def:alpha}} = 2\H^2(\partial B_\alpha) \geq 2\|\partial B_\alpha\|(D)
\stackrel{\pref{ineq:perimeter2}}> \int_0^{1-\alpha} f(\sigma)\, d\sigma,
\]
a contradiction. Therefore there exists an $r=:\beta  R$ satisfying~\pref{ineq:cutoff_r}, and the result follows as remarked above.
\end{proof}

{\bf Proof of Lemma~\ref{lemma:truncation}}:
Let $\alpha$ be as in Lemma \ref{lemma:alpha}. 
Choose $n$ balls $B^i$, of radius less than $1/8$, centered at $\{x^i\}_{i=1}^n$, and such that the family $\{2 B^i\}$ is disjoint. Set $w_k := a_k^{-1} u_k$, and note that for each $i$, 
\begin{align*}
\frac{1}{\H^2(\partial \alpha B^i)} \int_{B^i\setminus \overline{\alpha B^i}} |\nabla w_k|
  &\;+\; \dashint_{B^i\setminus \overline{\alpha B^i}} w_k
\;\leq\; \frac C{a_k} \left\{ \int_\Omega |\nabla u_k| + \int_\Omega u_k\right\},
\end{align*}
and this number tends to zero by~\pref{ass:small_perimeter}, implying that the function $w_k$ on $B^i$ is admissible for Lemma~\ref{lemma:alpha}. For each $i$ and each $k$, let $\beta_k^i$ be given by Lemma~\ref{lemma:alpha}, so that
\begin{equation}\label{lem-6.3-1}
2\|\partial \beta_k^i B^i\|(\supp u_k) \leq a_k^{-1}\|\nabla u_k\|(B^i\setminus \overline{\beta_k^i B^i}).
\end{equation}
Now set $\tilde u_k^i := u_k \chi_{\beta_k^iB^i}$ and $\tilde u_k := \sum_{i=1}^n \tilde u_k^i$. Then for any open $A\subset \Omega$ such that $x^i\in A$, 
\[
\liminf_{k\to\infty} \int _A \tilde u_k^i
\;=\; \liminf_{k\to\infty} \!\!\!\int\limits_{A\cap \beta_k^iB^i}\!\!\! u_k
\;\geq\;  \liminf_{k\to\infty}\!\!\! \int\limits_{A\cap \alpha B^i} \!\!\!u_k
\;\geq\; \sum_{j=1}^\infty m^j \delta_{x^j}(A\cap \alpha B^i)
\;\geq\; m^i,
\]
which proves~\pref{prop:tilde-u-k}; property~\ref{enum:lemma-truncation-2} follows from this by remarking that
\[
\limsup_{k\to\infty} \int_\Omega (u_k-\tilde u_k) = \lim_{k\to\infty} \int_\Omega u_k - \liminf_{k\to\infty} \int_\Omega \tilde u_k \leq \sum_{j=1}^\infty m^j - \sum_{j=1}^n m^j.
\]
The uniform separation of the supports is guaranteed by the condition that the family $\{2B^i\}$ is disjoint, and property~\ref{enum:lemma-truncation-1} follows by construction; it only remains to prove~\pref{ineq:perimeter-quantitative}.

For this we calculate
\begin{align}
\notag
\int_\Omega |\nabla \tilde u_k| 
&= \|\nabla u_k\|\Bigl(\bigcup_{i=1}^n \beta_k^i B^i\Bigr)
+ a_k \sum_{i=1}^n \|\partial \beta_k^i B^i\|(\supp u_k) \\
&\stackrel{\pref{lem-6.3-1}}\leq \int_\Omega |\nabla u_k| - \|\nabla u_k\|\Bigl(\Omega\setminus \bigcup_{i=1}^n \beta_k^i B^i\Bigr) + \frac12\sum_{i=1}^n \|\nabla u_k\|(B^i\setminus \overline{\beta_k^i B^i})\notag\\
&\leq \int_\Omega |\nabla u_k| - \frac12 \|\nabla u_k\|\Bigl(\Omega\setminus \bigcup_{i=1}^n \beta_k^i B^i\Bigr)\notag \\
&\leq \int_\Omega |\nabla u_k| - C_k\|u_k- \tdashint_{A_k} u_k \|_{L^{3/2}(A_k)}
\label{ineq:tilde-u}
\end{align}
Here the constant $C_k$ is the constant in the Sobolev inequality on the domain $A_k := \Omega\setminus \cup_i \overline{\beta_k^i B^i}$,
\[
C_k {\|u-\tdashint_{A_k} u \|_{L^{3/2}(A_k)} }\leq \frac12 \int_{A_k}|\nabla u|.
\]
The number $C_k>0$ depends on $k$ through the geometry of the domain $A_k$. Note that the size of the holes $\beta_k^i B^i$ is bounded from above by $B^i$ and from below by $\alpha B^i$. Consequently, for each $k_1$ and $k_2$ there exists a smooth diffeomorphism mapping $A_{k_1}$ into $A_{k_2}$, and the first and second derivatives of this mapping are bounded uniformly in $k_1$ and $k_2$. Therefore we can replace in~\pref{ineq:tilde-u} the $k$-dependent constant $C_k$ by a $k$-independent (but $n$-dependent) constant $C>0$.

Note that since $u_k$ is bounded in $L^1$,
\begin{equation}
\label{eq:au}
a_k^{-3/2}\|u_k\|_{L^{3/2}(A_k)}^{3/2} 
= a_k ^{-1} \|u_k\|_{L^1(A_k)} \to 0 \qquad\text{as }k\to\infty.
\end{equation}
Continuing from~\pref{ineq:tilde-u} we then estimate by the inverse triangle inequality
\begin{align*}
\notag
\int_\Omega |\nabla \tilde u_k| 
&\leq \int_\Omega |\nabla u_k| - C\|u_k \|_{L^{3/2}(A_k)} + \frac{C}{|A_k|^{1/3}}\|u_k\|_{L^1(A_k)}\\
&= \int_\Omega |\nabla u_k| - C\|u_k \|_{L^{3/2}(A_k)}  + \frac{C}{|A_k|^{1/3}a_k^{1/2}}\|u_k\|_{L^{3/2}(A_k)}^{3/2}\\
&= \int_\Omega |\nabla u_k| - C\|u_k \|_{L^{3/2}(A_k)}\left\{1-  \frac{1}{|A_k|^{1/3}a_k^{1/2}}\|u_k\|_{L^{3/2}(A_k)}^{1/2}\right\}\\
&\stackrel{\pref{eq:au}}\leq \int_\Omega |\nabla u_k| - C'\|u_k \|_{L^{3/2}(A_k)}
= \int_\Omega |\nabla u_k| - C'\|u_k -\tilde u_k\|_{L^{3/2}(\Omega)}.
\end{align*}
This proves the inequality~\pref{ineq:perimeter-quantitative}.
\qed

\bigskip

{\bf Proof of Lemma~\ref{lemma:other_sequence}}: 
By Lemma~\ref{lemma:compactness} and by passing to a subsequence we can assume that $v_\eta$ converges as measures to $v_0$. We first concentrate on the lower bound for $\Feb$. 

Fix $n\in\N$ for the moment. We apply Lemma~\ref{lemma:truncation} to the sequence $v_\eta$ and find a collection of components $v_\eta^i$, $i=1\dots n$, and $\tilde v_\eta = \sum_i  v_\eta^i$, such that 
\[
\wliminf_{\eta\to0} v^i_\eta \geq\sum_{i=1}^n m^i \delta_{x^i},
\]
and
\[
\int_{\T^3} |\nabla \tilde v_\eta|
\leq \int_{\T^3} |\nabla v_\eta| - C\|v_\eta-\tilde v_\eta\|_{L^{3/2}(\T^3)}.
\]
Setting $r_\eta := v_\eta-\tilde v_\eta$ we also have
\begin{align*}  \Hmospace {\tilde v_\eta}\Tb ^2 
 &= \int_\Tb\int_\Tb v_\eta(x) v_\eta(y)G_\Tb (x-y)\, dxdy
 - 2\int_\Tb \int_\Tb r_\eta(x)\tilde v_\eta(y) G_\Tb (x-y)\, dxdy\\
 &\qquad{}
 - \int\int r_\eta(x)r_\eta(y)G_\Tb (x-y)\, dxdy\\
&\leq \Hmospace {v_\eta} \Tb ^2  - 2\inf G_\Tb \|r_\eta\|_{L^1(\Tb)}\|\tilde v_\eta\|_{L^1(\Tb)}.
\end{align*}
Therefore 
\begin{equation}
\label{bound:other-seq:Feb}
\Feb(\tilde v_\eta) \leq \Feb (v_\eta) - C\eta \|r_\eta\|_{L^{3/2}(\T^3)}
  + C'\|r_\eta\|_{L^1(\Tb)}.
\end{equation}
Assuming the lower bound has been proved for $\tilde v_\eta$, we then find
\begin{align*}
\liminf_{\eta\to0} \Feb (v_\eta) 
&\geq \liminf_{\eta\to0} \Bigl[\Feb(\tilde v_\eta) - C'\|r_\eta\|_{L^1(\Tb)}\Bigr]\\
&\geq \Fzb\Bigl(\wliminf_{\eta\to0} \tilde v_\eta\Bigr)
 - C'\lim_{\eta\to0} \int_\Tb v_\eta + C' \liminf_{\eta\to0} \int_\Tb \tilde v_\eta\\
&\geq \Fzb\Bigl(\sum_{i=1}^n m^i\delta_{x^i}\Bigr) -C'\sum_{i=n+1}^\infty m^i.
\end{align*}
Taking the supremum over $n$ the lower  bound inequality for $v_\eta$ follows. 
 
\medskip
Turning to a lower bound for $\Heb$, we remark that  by Lemma~\ref{lemma:char_limit} the number of $x^i$ in~\pref{eq:structure_v0} with non-zero weight $m^i$ is finite. Choosing $n$ equal to this number, we have 
\[
\lim_{\eta\to0}\int_\Tb v_\eta \geq \liminf_{\eta\to0} \int_\Tb \tilde v_\eta 
\geq \sum_{i=1}^n\liminf_{\eta\to0} \int_\Tb v_\eta^i \geq \sum_{i=1}^n m^i = \lim_{\eta\to0} \int_\Tb v_\eta,
\]
and therefore $v_\eta^i\weakto m_0^i\delta_{x^i}$, and 
$\int_\Tb r_\eta \to0$. Then
\begin{align}
\notag
\Heb(\tilde v_\eta ) &= \frac1\eta\left[\Feb(\tilde v_\eta) - \fzb\left(\int_\Tb \tilde v_\eta\right)\right]\\
&\stackrel{\pref{bound:other-seq:Feb}}\leq \frac1\eta\left[\Feb(v_\eta) - \fzb\left(\int_\Tb  v_\eta\right)\right]
-C\|r_\eta \|_{L^{3/2}(\Tb)} + \frac{C'}\eta\|r_\eta\|_{L^1(\Tb)}
\notag \\
&  \qquad \qquad \qquad 
+ \,\, \frac1\eta\left[\fzb\left(\int_\Tb  v_\eta\right)-\fzb\left(\int_\Tb  \tilde v_\eta\right)\right]
\label{est:other-seq:Heb}\\
&\leq \Heb(v_\eta) -\frac {C}{\eta} \Bigl(\int_\Tb  r_\eta\Bigr)^{2/3}
+ \frac {L+C'}\eta \int_\Tb  r_\eta.
\notag
\end{align}
Here $L$ is an upper bound for $\fzb'$ on the set $[\inf_\eta \int \tilde v_\eta,\infty)$ (see Lemma~\ref{lemma:char_limit}) and in the passage to the last inequality we used the triangle inequality for $\| \cdot \|_{L^{3/2}}$ and the fact that by construction, $r_\eta$ takes on only two values.  
For sufficiently small $\eta$, the last two terms add up to a negative value, and therefore we again have $\Heb(\tilde v_\eta)\leq \Heb(v_\eta)$. Because of the choice of $n$ we have $\tilde v_\eta\weakto v_0$; if we assume, in the same way as above, that the lower bound has been proved for $\tilde v_\eta$,  we then find that
\[
\liminf_{\eta\to0} \Heb(v_\eta)\geq \liminf_{\eta\to0} \Heb(\tilde v_\eta)
\geq \Hzb(v_0).
\]
For use below we note that 
\begin{align}
\notag
\Heb(v_\eta)&\geq \Heb(\tilde v_\eta)  = \frac1\eta\left[\Feb(\tilde v_\eta) - \fzb\left(\int_\Tb \tilde v_\eta\right)\right] \\
\notag 
& \stackrel{\pref{e-0-subadd}}\geq\sum_{i=1}^n \left[\int_\Tb|\nabla  v^i_\eta| + \| v_\eta^i\|^2_{H^{-1}(\Tb)} - \frac1\eta\fzb\left(\int_\Tb  v^i_\eta\right)\right] 
+ 2\sum_{\substack{i,j=1\\i\not=j}}^n\int_\Tb\int_\Tb v_\eta^i(x)v_\eta^j(y) G_\Tb(x-y)\, dxdy\\
&\geq \sum_{i=1}^n \left[\int_{\R^3}|\nabla  v^i_\eta| + \| v_\eta^i\|^2_{H^{-1}(\R^3)} - \frac1\eta\fzb\left(\int_{\R^3}  v^i_\eta\right)\right] 
+ \inf_\Tb g^{(3)} \sum_{i=1}^n \left(\int_{R^3} v_\eta^i\right)^2\notag\\
&\qquad {}+ 2\inf G_\Tb \sum_{\substack{i,j=1\\i\not=j}}^n \int_{\R^3} v_\eta^i \int_{\R^3} v_\eta^j
\label{est:bdd-terms}
\end{align}
In the calculation above, and in the remainder of the proof, we switch to considering $v_\eta^i$ defined on $\R^3$ instead of $\Tb$. Since the terms in the first sum above are non-negative, boundedness of $\Heb(v_\eta)$ as $\eta\to0$ implies the boundedness of each of the terms in the sum independently.

We now show that when $\Heb(v_\eta)$ is bounded, then for each $i$ 
\begin{equation}
\label{property:xi}
\exists \xi^i_\eta\in\R^3: \qquad
\int_{\R^3} |x-\xi^i_\eta|^2 \,   v_\eta^i(x)\, dx = O(\eta^2)\qquad\text{as }\eta\to0.
\end{equation}
Suppose that this is not the case for some $i$; fix this $i$. We choose for $\xi_\eta$ the barycenter of $v_\eta^i$, i.e.
\begin{equation}
\label{def:xi_eta}
\xi_\eta = \frac{\displaystyle \int_{\R^3} x v_\eta^i(x)\, dx}{\displaystyle \int_{\R^3} v_\eta^i}.
\end{equation}
Since we assume the negation of~\pref{property:xi}, we find that
\begin{equation}
\label{ass:rho}
\rho_\eta^2 := \int_{\R^3} |x-\xi_\eta|^2 \, v_\eta^i(x)\, dx \gg \eta^2.
\end{equation}
Note that by~\pref{def:xi_eta} and the fact that $v_\eta^i\weakto x^i$,
\begin{equation}
\label{conv:rho_eta}
\lim_{\eta\to0}\rho_\eta = 0.
\end{equation}

Now rescale $v_\eta^i$ by defining $\zeta_\eta(x) : =\rho_\eta^3 v_\eta^i(\xi_\eta+\rho_\eta x)$. The sequence $\zeta_\eta$ satisfies
\begin{enumerate}
\item $\zeta_\eta\in BV(\R^3,\{0,\rho_\eta^3\eta^{-3}\})$;
\item $\displaystyle \int_{\R^3} \zeta_\eta = \int_{\R^3} v_\eta^i$;
\item $\displaystyle\frac\eta{\rho_\eta}\biggl(\int_{\R^3} |\nabla \zeta_\eta| +\|\zeta_\eta\|^2_{H^{-1}(\R^3)}\biggr) = \eta \int_{\R^3} |\nabla  v_\eta^i| + \eta \| v_\eta^i\|^2_{H^{-1}(\R^3)}$, and 
\item $\displaystyle\int_{\R^3} |x|^2 \, \zeta_\eta(x)\, dx  = 1$.
\end{enumerate}
The  first three properties imply that the sequence $\zeta_\eta$ is of the same type as the sequence $v_\eta$ in the rest of this paper, provided one replaces the small parameter $\eta$ by the small parameter $\tilde \eta := \eta/\rho_\eta$. The fourth property implies that the sequence is tight. By the third property above, \pref{est:bdd-terms}, and~\pref{conv:rho_eta}, the boundedness of $\Heb$ translates into the vanishing of the analogous expression for $\zeta_\eta$:
\begin{equation}
\label{limit:truncation-zeta}
\limsup_{\eta\to0} \;\Biggl\{\int_{\R^3} |\nabla \zeta_\eta| +\|\zeta_\eta\|^2_{H^{-1}(\R^3)} - \frac{\rho_\eta}\eta \fzb\Bigl(\int_{\R^3}\zeta_\eta\Bigr)\Biggr\} = 0.
\end{equation}
We now construct a contradiction with this limiting behavior, and therefore prove~\pref{property:xi}.

\medskip
Following the same arguments as for $v_\eta$ we apply the concentration-compactness lemma of Lions~\cite{L1} to find that the sequence $\zeta_\eta$ converges to (yet another) weighted sum of delta functions
\[
\mu := \sum_{j=1}^\infty \frak m^j \delta_{y^j}
\qquad\text{with}\qquad
\mu(\R^3) = m^i,
\]
where $\frak m^j\geq0$ and $y^j\in \R^3$ are distinct. Since 
\[
\int x \, d\mu(x) = \lim_{\eta\to0} \int_{\R^3} x \, \zeta_\eta(x)\, dx = 0
\qquad\text{and}\qquad
\int |x|^2 \, d\mu(x) = \lim_{\eta\to0} \int_{\R^3} |x|^2\, \zeta_\eta(x)\, dx = 1,
\]
at least two different $\frak m^j$ are non-zero; we assume those to be $j=1$ and $j=2$.  

We will need to show that the number of non-zero $\frak m^j$ is finite. Assuming the opposite for the moment, choose $n\in\N$ so large that 
\[
\sum_{j=1}^n \fzb(\frak m^j) > \fzb(m^i);
\]
this is possible since there exist no minimizers for $\fzb(m^i)$ with infinitely many non-zero components (Lemma~\ref{lemma:char_limit}).  We apply Lemma~\ref{lemma:truncation} to find a new sequence 
$\tilde \zeta_\eta= \sum_{j=1}^n \zeta_\eta^j$, where $\zeta_\eta^j\weakto \frak m^j\delta_{y^j}$.
Then 
\begin{align*}
\liminf_{\eta\to0} \frac\eta{\rho_\eta}\biggl\{\int_{\R^3}|\nabla \zeta_\eta| + \|\zeta_\eta\|^2_{H^{-1}(\R^3)}\biggr\}
&\geq\liminf_{\eta\to0}  \sum_{j=1}^n \frac\eta{\rho_\eta}\biggl\{\int_{\R^3}|\nabla \zeta^j_\eta| + \|\zeta^j_\eta\|^2_{H^{-1}(\R^3)}\biggr\}
\\
&\geq
 \liminf_{\eta\to0}  \sum_{j=1}^n \fzb\Bigl(\int_{\R^3} \zeta_\eta^j\Bigr) 
> \fzb(m^i),
\end{align*}
which contradicts~\pref{limit:truncation-zeta}; therefore the number of non-zero components $\frak m^j$ is finite, and we can choose $n$ such that $m^i = \sum_{j=1}^n \frak m^j$ and $\int(\zeta_\eta-\tilde\zeta_\eta)\to0$.

To conclude the proof we now note that
\begin{align*}
\int_{\R^3} |\nabla \zeta_\eta| &+\|\zeta_\eta\|^2_{H^{-1}(\R^3)}-  \, \frac{\rho_\eta}\eta\fzb\Bigl(\int_{\R^3} \zeta_\eta\Bigr)\\
&\geq \sum_{j=1}^n \Bigl\{\int_{\R^3} |\nabla \zeta_\eta^j| +\|\zeta_\eta^j\|^2_{H^{-1}(\R^3)}\Bigr\}
 + 2(\zeta_\eta^1,\zeta_\eta^2)_{H^{-1}(\R^3)} + C\|\zeta_\eta-\tilde\zeta_\eta\|_{L^{3/2} (\R^3)}\\
 & \qquad - \,\frac{\rho_\eta}\eta\fzb\Bigl(\int_{\R^3} \zeta_\eta\Bigr)\\
 &\geq \frac{\rho_\eta}\eta\Biggl[\sum_{j=1}^n \fzb\Bigl(\int_{\R^3} \zeta_\eta^j\Bigr) 
 - \fzb\Bigl(\int_{\R^3} \tilde \zeta_\eta\Bigr)\Biggr]
+ \frac{\rho_\eta}\eta\Biggl[\fzb\Bigl(\int _{\R^3}\tilde \zeta_\eta\Bigr) - \fzb\Bigl(\int_{\R^3} \zeta_\eta\Bigr)\Biggr]\\
&\qquad{} + 2(\zeta_\eta^1,\zeta_\eta^2)_{H^{-1}(\R^3)} + C\|\zeta_\eta-\tilde\zeta_\eta\|_{L^{3/2}(\R^3)}\\
&\geq -\frac{L\rho_\eta}\eta \int_{\R^3} (\zeta_\eta-\tilde\zeta_\eta)+ \frac{C\rho_\eta}\eta \Bigl(\int_{\R^3} (\zeta_\eta-\tilde\zeta_\eta)\Bigr)^{2/3}
+ \frac1{2\pi} \int_{\R^3}  \int_{\R^3}  \frac{\zeta_\eta^1(x)\zeta_\eta^2(y)}{|x-y|}\, dxdy.
\end{align*}
Since $\lim_{\eta\to0}\int_{\R^3} (\zeta_\eta-\tilde\zeta_\eta) = 0$, the first two terms in the last line above eventually become positive; the final term converges to $(2\pi)^{-1}\frak m^1\frak m^2|y^1-y^2|^{-1}>0$. This contradicts~\pref{limit:truncation-zeta}.
\qed

\subsection{Proof of Lemma~\ref{lemma:char_limit}}
Let $z_n$ be a minimizing sequence for $\fzb(m)$.
The functions
\[
z_n^\e(x) := z_n\left(\frac{x}{(1+\e/m)^{1/3}}\right)
\]
are admissible for $\fzb(m+\e)$ for all $\e>-m$. Since the functions
\[
f_n(\e) := \int_{\R^3} |\nabla z_n^\e| + \|z_n^\e\|_{H^{-1}(\R^3)}^2
= (1+\e/m)^{2/3}\int_{\R^3} |\nabla z_n| + (1+\e/m)^{5/3}\|z_n\|_{H^{-1}(\R^3)}^2
\]
satisfy
\begin{multline}
\label{exp:f_n}
f_n(\e) = f_n(0) + \frac\e m\left(\frac{2}3 \int_{\R^3}|\nabla z_n| + \frac53 \|z_n\|_{H^{-1}(\R^3)}^2\right)\\
+ \frac {\e^2}{2m^2} \left(-\frac{2}9 \int_{\R^3}|\nabla z_n| + \frac{10}9 \|z_n\|_{H^{-1}(\R^3)}^2\right) + O\Bigl(\frac{\e}m\Bigr)^3,
\end{multline}
uniformly in $n$, we have for all $\e\geq0$,
\begin{equation}
\label{ineq:upper_bound_fz}
\fzb(m+\e) \leq \inf_n f_n(\e) \leq  \fzb(m) + \frac53 \fzb(m)\frac{\e} m + \frac{5}9 \fzb(m)\Bigl(\frac\e m\Bigr)^2 + O\Bigl(\frac{\e}m\Bigr)^3,
\end{equation}
We deduce that
\begin{equation}\label{fzb-bound}
\fzb(m+\e)-\fzb(m) \leq \frac 5{3m}\fzb(m) \, \e + O(\e^2). 
\end{equation} 
By (\ref{e-0-subadd}), we find that  for any $m \geq 1$ and any positive integer $n$, we have 
\[ \fzb (m)  \leq \fzb (1) \, + \, n \fzb \left( \frac{m-1}{n} \right). \]
By taking $n$ such that  $ \frac{m-1}{n} \in [1,2]$, we have 
\[ \fzb (m)  \leq \fzb (1) \, + \, C n \]
where $C$ denotes a uniform bound for $\fzb$ on the interval $[1,2]$. By choice of $n$ we have for some constant $C'$, $  
\fzb (m)  \leq \fzb (1) \, + \, C' m$.   
Combining this with (\ref{fzb-bound}), we find that   $\fzb'$ is bounded from above on sets of the form $[a,\infty)$ with $a>0$.

For the concaveness of $\fzb$, note that under a constant-mass constraint $\int |\nabla z|$ is minimal for balls and $\|z\|_{H^{-1}(\R^3)}$ is maximal for balls (see e.g.~\cite{Burchard96} for the latter). Setting $m = \int z$ and $r^3=3m/4\pi$, we therefore have
\[
-\frac{2}9 \int_{\R^3}|\nabla z_n| + \frac{10}9 \|z_n\|_{H^{-1}(\R^3)}^2
\leq -\frac{2}9 \int_{\R^3}|\nabla \chi^{}_{B_r}| + \frac{10}9 \|\chi^{}_{B_r}\|_{H^{-1}(\R^3)}^2,
\]
and an explicit calculation shows that the right-hand side is negative iff $m< 2\pi $. From~\pref{exp:f_n} we therefore have for all $m<2\pi$ and all $\e>-m$,
\[
\fzb(m+\e)\leq \fzb(m) + \inf_n \Bigl[a_n \e - b\e^2 + c\e^3\Bigr],
\]
where $a_n$ is a sequence of real numbers, and $b,c>0$. Writing this as
\[
\fzb(m) \leq \inf_{m_0\in(0,2\pi)}\left\{ 
  \fzb(m_0) + \inf_n\Bigl[a_n (m-m_0) - b(m-m_0)^2 + c(m-m_0)^3\Bigr]\right\},
\]
we note that for each $m_0$ the expression in braces is strictly concave in $m$ for $|m-m_0|< b/3c$; since the infimum of a set of concave functions is concave, it follows that the right-hand side is a concave function of $m$. Since equality holds for $m_0=m$, $\fzb$ is therefore concave for $m\leq 2\pi $, and $\fzb''(m)<0$ for $m<2\pi $.

Finally, part 3 follows from remarking that if (say) $m^1, m^2\in(0,2\pi )$, then
\[
\frac{d^2}{d\e^2} \Bigl(\fzb(m^1+\e) + \fzb(m^2-\e)\Bigr)\Bigr|_{\e=0}
= \fzb''(m^1) + \fzb''(m^2)<0.
\]
Therefore the sequence $(m^1,m^2,\ldots)$ is not optimal, a contradiction. It follows that there can be at most one $m^i$ in the region $(0,2\pi )$, and since the  total sum is finite, the number of non-zero $m^i$ is finite.
\qed


\section{Two dimensions}
\label{sec:2d}
All differences between the two- and three-dimensional case arise from a single fact: the scaling of the $H^{-1}$ is critical in two dimensions, making the two-dimensional case special. 

\subsection{Leading-order convergence}
The first difference is encountered in the leading-order limiting behavior. As we discussed in Section~\ref{sec:degeneration}, the leading-order contribution to the $H^{-1}$-norm involves the masses of the particles instead of their localized $H^{-1}$-norm (see~\pref{eq:local_L1Hmo}). For the local problem in two dimensions we therefore introduce
the function 
\begin{align}
\fza(m) &:= \frac{m^2}{2\pi} + \inf\left\{\int_{\R^2} |\nabla z|: 
   z\in BV(\R^2;\{0,1\}), \int_{\R^2} z = m \right\}\label{def:fe}\\
& = \frac{m^2}{2\pi} + 2 \sqrt{\pi m}.
\notag
\end{align}
Note that the minimization problem in~\pref{def:fe} is simply to minimize 
perimeter for a given area, and a disc of the appropriate area is the only solution.  
Thus  the value of $\fza(m)$ can be determined explicitly.

The function $\fza$ does not satisfy the lower-semicontinuity condition~\pref{e-0-subadd} ({\it cf.} Remark \ref{remark-lsc}). We therefore introduce the lower-semicontinuous envelope function 
\begin{equation}
\label{def:lscfz}
\lscfza(m) := \inf \left\{\sum_{j= 1}^\infty \fza(m^j): m^j\geq 0,\,\,  \sum_{j=1}^\infty m^j = m \right\}. 
\end{equation}
The limit functional is defined in terms of this envelope function:
\[ 
\Fza(v) := \begin{cases}
\sum_{i=1}^\infty \lscfza(m^i) & \text{if } v = \sum_{i=1}^\infty m^i \delta_{x^i} 
\text{ with } \{x^i\} \text{ distinct and }m^i\geq0\\
\infty & \text{otherwise}.
\end{cases}
\]

\begin{theorem}\label{first-order -limit}
Within the space $X$, we have 
\[
\Fea \, \stackrel{\Gamma}\longrightarrow \, \Fza \qquad {\rm as} \quad \eta \rightarrow 0.
\]
That is,  conditions $1$ and $2$ of Theorem $\ref{3D-first-order -limit}$ hold with $\Feb$ and $\Fzb$ replaced by 
$\Fea$ and $\Fza$. 
\end{theorem}

The proof follows along exactly the same lines as the proof of Theorem~\ref{3D-first-order -limit}. It is in fact simpler, since a standard result on the approximation for  sets of finite perimeter (see for example Theorem  3.42 of \cite{AFP}) implies that, without loss of generality, we may assume  that a sequence $v_\eta$ with bounded energy (for $\eta$ sufficiently small) satisfies 
\begin{equation}
\label{split:v_eta-2d}
v_\eta = \sum_{i=1}^\infty v_\eta^i \quad {\rm with} \quad v_\eta^i = \frac{1}{\eta^{2}}\, \chi_{A_{\eta}^i}, 
\end{equation}
where the sets $A_{\eta}^i$ are connected, disjoint, smooth, and with diameters which tend to zero as $\eta \rightarrow 0$. Then the following estimate holds true in two dimensions: 
\begin{equation}\label{2d-conc-est}
\sum_{i=1}^\infty \diam(\supp v_\eta^i)
\, \leq \, \eta^2 \, \sum_{i=1}^\infty \int_\Ta |\nabla v_\eta^i|
\, \leq \, {\eta}\Fea(v_\eta) = O(\eta), 
\end{equation}
which can be used to bypass Lemma \ref{lemma:other_sequence}.

\subsection{Next-order behavior}
Turning to the next-order behavior, note that  
among all measures of mass $M$, the global minimizer of $\Fza$ is given by
\[
\min \left\{\Fza(v): \int_{\Ta} v = M\right\} = \lscfza(M).
\]
We recover the next term in the expansion as the limit of $\Fea - \lscfza$, appropriately rescaled, that is of the functional
\[
\Hea(v) := \left|\log\eta\right| \left[\Fea(v) - \lscfza\left(\int_\Ta v\right)\right].
\]

Here the situation is similar to the three-dimensional case in that for boundedness of the sequence $\Hea$ the limiting weights $m^i$ should satisfy two requirements: a minimality condition and  a compactness condition. The compactness condition is most simply written as the condition that
\begin{equation}
\label{cond:compactness-2d}
\lscfza(m^i) = \fza(m^i)
\end{equation}
and corresponds to the condition in three dimensions that there exist a minimizer of the minimization problem~\pref{def:fe3}.

In two dimensions, the minimality condition~\pref{pb:equalmass} provides a characterization that is stronger than the in three dimensions: 
\begin{lemma}
\label{lemma:charmin}
Let $\{m^i\}_{i\in \N}$ be a solution of the minimization problem
\begin{equation}
\label{pb:equalmass}
\min \left\{\sum_{i=1}^\infty \fza(m^i): m^i\geq 0, \ 
   \sum_{i=1}^\infty m^i = M.\right\}.
\end{equation}
Then only a finite number of the terms  $m^i$ are non-zero and all the non-zero terms are equal. In addition, if one $m^i$ is less than $2^{-2/3}\pi$, then it is the only non-zero term.  
\end{lemma}
The proof is presented in Section \ref{sec:proofs}.
We will also need the following corollary on the stability of $\Fza$ under perturbation of mass:
\begin{corollary}
\label{cor:lscfz-cont}
The function $\lscfza$ is Lipschitz continuous on $[\delta,1/\delta]$ for any $0<\delta<1$.
\end{corollary}

\bigskip 

The limit as $\eta\to0$ of the functional $\Hea$ has one additional term in comparison to the three-dimensional case, which arises from the second term in~\pref{eq:rewrite-Hmo2},\begin{equation}
\label{local_terms}
-\frac1{2\pi}\sum_{i=1}^\infty \int_{\R^2}\int_{\R^2} z_\eta^i(x)z_\eta^i(y) \log |x-y|\, dxdy .
\end{equation}
To motivate the limit of this term, recall that~$z^i_\eta$ appears in the minimization problem~\pref{def:fe}, which has only balls as solutions. Assuming~$z^i_\eta$ to be a characteristic function of a ball of mass $m^i$, we calculate that the first term in~\pref{local_terms} has the value $f_0(m^i)$, where 
\[
 f_0(m) := \frac{m^2}{8\pi} \left(3-2\log\frac m\pi\right).
\]

We therefore define the intended $\Gamma$-limit $\Hza$ of $\Hea$ as follows. First let us introduce some notation: for $n\in \N$ and $m>0$ the sequence $n\otimes m$ is defined by
\[
(n\otimes m)^i := \begin{cases}
  m & 1\leq i\leq n\\
  0 & n+1 \leq i < \infty.
\end{cases}
\]
Let $\widetilde{\mathcal M}$ be the set of optimal sequences for the problem~\pref{pb:equalmass}:
\[
\widetilde{\mathcal M} := \left\{ n\otimes m : n\otimes m \text{ minimizes }\pref{pb:equalmass} \text{ for }M= nm, \text{ and }\lscfza(m)=\fza(m)\right\}.
\]
Then define
\begin{equation}
\label{def:Hz}
\Hza(v): = \begin{cases}
\displaystyle n \left\{ f_0(m) \, +\,  m^2\, g^{(2)} (0) \right\}\,  +& \\
  \qquad {} \,\,  
  \displaystyle\frac {m^2}2  \sum_{\substack{i,j\geq 1\\ i\not=j}} G_\Ta(x^i-x^j)
     &\displaystyle\text{if } v = m\sum_{i=1}^n \delta_{x^i}, \,\, \{x^i\} \text{ distinct}, \, \, n\otimes m \in \widetilde{\mathcal M},\\
\infty & \text{otherwise}.
\end{cases}
\end{equation}
\medskip
\begin{theorem}
\label{th:sharpnextlevel}
Within the space $X$, we have 
\[
\Hea \stackrel{\Gamma}\longrightarrow \Hza \qquad {\rm as} \quad \eta \rightarrow 0. 
\]
That is,  Conditions $1$ and $2$ of Theorem $\ref{first-order -limit}$ hold with $\Fea$ and $\Fza$ replaced with 
$\Hea$ and $\Hza$ respectively. 
\end{theorem}

The proof of this theorem again closely follows that of Theorem~\ref{3D-th:sharpnextlevel}. The compactness property~\pref{cond:compactness-2d} in the lower bound follows by a simpler argument than in three dimensions, however. Using the division into components with connected support~\pref{split:v_eta-2d}, we have
\begin{align}
\Hea(v_\eta) &= \left|\log\eta\right| \left[\Fea(v_\eta) - \lscfza\left(\int_\Ta v_\eta\right)\right]
\nonumber
\\
&=\logeta\sum_{i=1}^\infty \left[\int_{\R^2} |\nabla z_\eta^i| + \frac1{2\pi} \Bigl(\int_{\R^2}z_\eta^i\Bigr)^2- \fza\Bigl(\int_{\R^2}z_\eta^i\Bigr)\right]\nonumber\\
\label{ineq:H1}
&\qquad{}+ \logeta \sum_{i=1}^\infty \left[\fza\Bigl(\int_{\R^2}z_\eta^i\Bigr)
   - \lscfza\left(\int_{\R^2} z^i_\eta\right)\right]\\
\label{ineq:H2}&\qquad{}+ \logeta \left[ \sum_{i=1}^\infty \lscfza\Bigl(\int_{\R^2}z_\eta^i\Bigr)
   - \lscfza\left(\int_\Ta v_\eta\right)\right] \\
&\qquad{}+\sum_{i=1}^\infty \left\{
  -\frac1{2\pi} \int_{\R^2}\int_{\R^2} z_\eta^i(x)\, z_\eta^i(y) \, 
      \log|x-y|\, dx\, dy
  + \int_\Ta \int_\Ta v_\eta^i(x) \, v_\eta^i(y) \, g^{(2)} (x-y)\, dx\, dy \right\}\notag \\
  &\qquad {}+  \sum_{\substack{i,j = 1 \\i\not=j}}^\infty
    \int_\Ta \int_\Ta v_\eta^i(x) \,  v_\eta^j(y) \, G_\Ta(x-y)\, dx\, dy\label{ineq:H}.
\end{align}
The last two lines in the development above are uniformly bounded from below. Since $\Hea(v_\eta)$ is bounded from above, it follows that the terms in square brackets, which are non-negative, tend to zero. In combination with the continuity of $\fza$ and $\lscfza$ this implies the compactness property~\pref{cond:compactness-2d}.
We also remark that because  the contents of the square brackets in 
(\ref{ineq:H1}) and  (\ref{ineq:H2}) are zero in the limit, we find with the aid of Lemma 
\ref{lemma:charmin} that the number of concentration points $x^i$ in the weak limit of $v_\eta$ is finite with equal coefficient weights. Moreover, 
we may assume that there are a finite number of different components of $v_\eta$, and each must converge to a different $x^i$; otherwise, the last term in (\ref{ineq:H})   would tend to $\infty$  as $\eta$ tends to $0$.    

\subsection{Proofs of Lemma \ref{lemma:charmin} and Corollary~\ref{cor:lscfz-cont}}
\label{sec:proofs}

The proof of Lemma~\ref{lemma:charmin}  contains two elements. The first element is general, and only uses the property that $\fza$ is concave on $\left[0, \frac{\pi}{\sqrt[3]{4}}\right] $ and convex on  $\left[\frac{\pi}{\sqrt[3]{4}}, \infty\right) $. This property reduces the possibilities to a combination of (a) a finite number of equal $m^i$ in the convex region with possibly (b) one $m^i$ in the concave region (see~\cite[Section~5.4]{LP} for a similar reasoning). The second part, in which possibility (b) above is excluded, depends heavily on the exact form of~$\fza$, and is an uninspiring exercise in estimation. \\

{\bf Proof of Lemma~\ref{lemma:charmin}}:  
For this proof only,  let us abuse notation and  use $x, x^i, y, z$   to denote 
positive real numbers. 
We note that 
\[
\fza(m) = 2^{5/3} \pi   \, f\left(\frac{m}{\pi \, 2^{4/3} }\right)
\qquad\text{with}\qquad
f(x) = x^2 + \sqrt{ x}.
\]
We therefore continue with $f$ instead of $\fza$. 
Since $f$ is concave on $\left(0,\frac{1}{4}\right]$ and convex on $\left[\frac{1}{4},\infty\right)$, the following hold true: 
\begin{itemize}
\item There is at most one $x^i\in \left(0,\frac{1}{4}\right)$; for if $x^i,x^j\in \left(0,\frac{1}{4}\right)$, then
\[
\frac {d^2}{d\e^2} (f (x^i+\e) + f (x^j-\e))\Bigr|_{\e=0}
 \,  =\,  f''(x^i) + f''(x^j) < 0,
\]
contradicting minimality. Therefore only one non-zero element is less than $\frac{1}{4}$, which also implies that the number of non-zero elements is finite.
\item The set of elements $\left\{x^i: x^i\geq \frac{1}{4}\right\}$ is a singleton, since the function is convex on $\left[\frac{1}{4},\infty\right)$.
\end{itemize}

Therefore the lemma is proved if we can show the following.
Take any sequence of the form
\begin{equation}\label{testpoint}
x^i = \begin{cases}
  x & i = 1\\
  y & i = 2, \dots, n+1\\
  0 & i\geq n+2,
\end{cases}
\end{equation}
with $x<1/4\leq y$; then this sequence can {\it not} be a solution of the minimization problem~\pref{pb:equalmass}.

To this end, we first note that
\[
(n+1)f\left(\frac n{n+1} \, y\right) - nf(y) =
  \frac n{n+1}\sqrt y \left(-y^{3/2} + (n+1)\left(\sqrt{\frac{n+1}n}-1\right)\right).
\]
If this expression is negative, then by replacing the $n$ copies of $y$ in (\ref{testpoint}) 
by $n+1$ copies of $ny/(n+1)$ we decrease the value in~\pref{pb:equalmass}. Therefore we can assume that 
\[
\frac14 \leq y\leq y_m(n) := (n+1)^{2/3}\left(\sqrt{\frac{n+1}n}-1\right)^{2/3}.
\]

We distinguish two cases. \textbf{Case one:} If $y+x/n < y_m(n)$, then we compare  our sequence (\ref{testpoint}) with $n$ copies of $z := y+x/n$:
\begin{align*}
f(x)+nf(y) - nf(y+x/n) &= f(x) +nf(z-x/n) - nf(z)\\
&= x^2 \Bigl(1+\frac1n\Bigr) - 2xz + \sqrt x + n\sqrt{z-\frac xn} - n \sqrt z\\
&=: g(x,z).
\end{align*}
We now show that $g$ is strictly positive for all relevant values of $x$ and $z$, i.e. for $0<x<1/4$ and $1/4+1/n< z< y_m(n)$. 

Differentiating $g(x,z)/x$ we find that
\begin{equation}
\label{eq:derivgx}
\frac \partial{\partial x} \frac{g(x,z)}x 
= 1+\frac1n - \frac 1{2x^{3/2}} 
  - \frac n{x^2}\left( \sqrt{z-\tfrac xn} - \sqrt z\right)
  - \frac1{2x\sqrt{z-\tfrac xn}}.
\end{equation}
This expression is negative: $x<1/4$ implies that
\[
1+\frac 1n - \frac1{2x^{3/2}} < 0,
\]
and by concavity of the square root function we have  
\[
\sqrt{z} \leq \sqrt{z-\frac xn} +\frac1{2\sqrt{z-\tfrac xn}}\, \frac xn,
\]
so that the last two terms in~\pref{eq:derivgx} together are also negative. 

Since $g(x,z)/x$ is decreasing in $x$, it is bounded from below by 
\[
4g(1/4,z) = \frac14 \left(1+\frac1n\right) -2z + 2
  + 4n \left( \sqrt{z-\tfrac1{4n}} - \sqrt z\right).
\]
The right-hand side of this expression is concave in $z$, and therefore bounded from 
below by the values at $z=(1+1/n)/4$ and at $z=y_m(n)$. The first of these is  
\[
-\frac14\left(1+\frac1n\right) + 2 + 2n\left(1-\sqrt{1+\tfrac1n}\right)
\geq -\frac 12 + 2 + 2n\left(1 - (1-1/2n)\right) = \frac12.
\]
For the second, the expression 
\[
4g(1/4,y_m(n)) = \frac14 \left(1+\frac1n\right) -2y_m(n) + 2
  + 4n \left( \sqrt{y_m(n)-\tfrac1{4n}} - \sqrt {y_m(n)}\right)
\]
is positive for $n=1,2$, as can be checked explicitly; for $n\geq 3$, we estimate $2^{-2/3}\leq y_m(n)\leq ((n+1)/2n)^{2/3}$ and therefore
\begin{align*}
&\frac14 \left(1+\frac1n\right) -2y_m(n) + 2
  + 4n \left( \sqrt{y_m(n)-\tfrac1{4n}} - \sqrt {y_m(n)}\right)\\
&\qquad{}\geq \frac14 - 2\left(\frac{n+1}{2n}\right)^{2/3}
    + 2 - \frac1{2\sqrt{y_m(n) - \tfrac1{4n}}}\\
&\qquad {}\geq \frac94 - 2\left(\frac{n+1}{2n}\right)^{2/3}
  - \frac1{2\sqrt{2^{-2/3} - \tfrac1{12}}}
\end{align*}
The right-hand side of this expression is strictly positive for all $n\geq 3$. This concludes the proof of case one.

For \textbf{case two} we assume that $y+x/n\geq y_m(n)$, set
\[
z := \frac {ny+x}{n+1}, 
\]
and compare the original structure with $n+1$ copies of $z$:
\begin{align*}
f(x)+nf(y) - (n+1)f(z) &= f(x) +nf\left(\frac{n+1}nz-\frac xn\right) - (n+1)f(z)\\
&= \frac{n+1}n (z-x)^2 + \sqrt x + \sqrt n\sqrt{({n+1})z-x} - (n+1) \sqrt z\\
&=: h(x,z).
\end{align*}
Note that the admissible values for $z$ are
\begin{equation}
\label{def:z-bound}
\frac n{n+1} y_m(n) \leq z \leq \frac{ny_m(n)+x}{n+1} \leq y_m(n).
\end{equation}

We first restrict ourselves to $n\geq2$, and state an intermediary lemma:
\begin{lemma} 
\label{lemma:lowerbound-h}
Let $n\geq 2$. Then for all $0<x<1/4$ and for all $z$ satisfying~$\pref{def:z-bound}$, 
\[
h(x,z) > \min\{h(0,z),h(1/4,z)\}.
\]
\end{lemma}

Assuming this lemma for the moment, we first remark that $h(0,z) \geq  0$ by the bound $z\geq ny_m(n)/(n+1)$ and the definition of $y_m$. For the other case we remark that the function
\[
n\mapsto \sqrt n\sqrt{({n+1})z-x} - (n+1) \sqrt z
\]
is increasing in $n$ for fixed $z$. Keeping in mind that $n\geq 2$ we therefore have
\[
h(1/4,z)\geq \frac32 \left(z-\tfrac14\right)^2 + \frac12 + \sqrt 2 \sqrt{3z-\tfrac14}
- 3\sqrt z,
\]
and this function is positive for all $z\geq 2y_m(2)/3\approx 0.51$. This concludes the proof for $n\geq 2$. 

\medskip
Before we prove Lemma~\ref{lemma:lowerbound-h} we first discuss the case $n=1$, for which
\[
h(x,z) = 2(z-x)^2 + \sqrt x + \sqrt {2z-x} - 2\sqrt z.
\]
The domain of definition of $z$ is 
\[
\left[\frac12 y_m(1),y_m(1)\right] \subset [0.4410,0.8821].
\]
The mixed derivative $h_{zx}$ is negative on the domain of $x$ and $z$, so that
\[
h_z(x,z) \geq h_z(1/4,z) = 4z-1 + \frac1{\sqrt{2z-1/4}} - \frac1{\sqrt z}.
\]
This expression is again positive for the admissible values of $z$, and we find
\[
h(x,z)\geq h\left(x,\tfrac 12 y_m(1)\right)
= 2\left(\tfrac12 y_m(1)-x\right)^2 + \sqrt x + \sqrt{y_m(1)-x} - 2\sqrt{\tfrac12 y_m(1)}.
\]
Similarly this expression is non-negative for all $0\leq x\leq 1/4$, which concludes the proof for the case $n=1$. 
\qed

\bigskip
{\bf Proof of Corollary \ref{cor:lscfz-cont}}: 
Fix $0<\delta< 2^{-2/3}\pi  $ and $M\in [\delta,1/\delta]$; by Lemma~\ref{lemma:charmin} there exist $n,m$ with $M=nm$ such that $\lscfza(M) = n\, \fza(m)$. Note that if $n=1$ then $m=M\geq \delta$, and if $n>1$ then by Lemma~\ref{lemma:charmin} $m\geq 2^{-2/3}\pi > \delta$; therefore we have $m\geq \delta$ and $n\leq M/\delta$. Since $\lscfza$ is the pointwise minimum of a collection of functions $\fza$, local Lipschitz continuity of $\lscfza$ now follows from the same property for the functions $\fza$ on the domain $[\delta,1/\delta]$. 
\qed

\bigskip

We still owe the reader the proof of Lemma~\ref{lemma:lowerbound-h}. \\
{\bf Proof of Lemma~\ref{lemma:lowerbound-h}}: 
We first show that if $4/25\leq x\leq 1/4$, then $h_x(x,z)< 0$. We estimate the derivative by using the bounds on $z$ and $x$:
\begin{align*}
h_x(x,z) &= -2\frac{n+1}n(z-x) - \frac{\sqrt n}{2\sqrt{(n+1)z-x}}
  + \frac1{2\sqrt x}\\
&\leq -2y_m(n) + \frac{n+1}{2n} - \frac{\sqrt n}{2\sqrt{ny_m(n)}} + \frac54.
\end{align*}
Note that $y_m$ is monotonically decreasing in $n$, and that we can estimate $y_m$ from below by
\[
y_m(n)^{3/2} = (n+1)\left(\sqrt{1+\tfrac1n}-1\right) 
  \geq (n+1)\, \frac1{2\sqrt{\frac{n+1}n}}\, \frac1{n}
  = \frac12\sqrt{\frac{n+1}n}.
\] 
Using $n\geq 2$ we find
\[
h_x(x,z) \leq -2\cdot 2^{-2/3}\left(\frac{n+1}n\right)^{1/3}
  + \frac{n+1}{2n} - \frac1{2\sqrt{y_m(2)}} + \frac54 
  =: \ell\left(\frac{n+1}n\right).
\]
The function $\ell$ is increasing on $[1,\infty)$, and we have 
\[
\ell\left(\frac{n+1}n\right) \leq \ell(\tfrac 32) < 0.
\]

On the remaining region $0<x<4/25$ the second derivative $h_{xx}$ is negative:
\begin{align*}
h_{xx}(x,z) &= 2\frac{n+1}n - \frac{\sqrt n}{4\bigl((n+1)z-x\bigr)^{3/2}}
   - \frac1{4 x^{3/2}}\\
&\leq 3 - \frac14 \frac{125}8 < 0.
\end{align*}
For any fixed $z$, therefore, the function $x\mapsto h(x,z)$ takes its minimum on the boundary, that is in one of the two points $x=0$ and $x=1/4$. Since the first derivative is non-zero on $[4/25,1/4]$, and since the second derivative is non-zero on $(0,4/25]$, the minimum is only attained on the boundary.
This concludes the proof of Lemma~\ref{lemma:charmin}. 
\qed

\section{Discussion}
\label{sec:discussion}

The results of this work provide a rigorous connection between the detailed, micro-scale model defined by $\cal E$ in~\pref{def:Fes} and the macroscopic, upscaled models given by the limiting energies $\Fza$, $\Hza$, $\Fzb$, and $\Hzb$. We now discuss some related aspects. 

\medskip

\textsl{Differences between the two- and three-dimensional cases: scaling.} The consequences of the difference between two and three dimensions in the scaling of the $H^{-1}$-norm are best appreciated in the Green's functions in the whole space: if we replace $x$ by $\eta x$, then 
\[
\log \eta x = \log \eta + \log x \qquad\text{in two dimensions, and}\qquad
\frac1{|\eta x|} = \frac1\eta\cdot \frac1{|x|} \qquad\text{in three.}
\]
The difference between the additive effects in two dimensions and the multiplicative effect in three dimensions is responsible for the difference between the two limiting problems:
\[
\int |\nabla z| + \left(\int z\right)^2\qquad\text{in two dimensions, and}\qquad
\int |\nabla z| + \|z\|_{H^{-1}}^2 \qquad\text{in three.}
\]

\medskip

\textsl{Differences between the two- and three-dimensional cases: local problems.}
Because of this difference in scaling, the local energy contributions $\fza$ (and $\lscfza$) and $\fzb$ are necessarily different, and since the two-dimensional local problem is the isoperimetric problem, its solution can be calculated explicitly in terms of $m$. For the three-dimensional local problem we can only conjecture on the structure of minimizers (see below). 

In addition to this, there is a difference in the handling of the lower semicontinuity in two and three dimensions. This comes from the fact that the definition of $\fza$ presupposes that the mass of $z$ remains localized (does not escape to infinity) while the definition of $\fzb$ does not. As a result, the function $\fzb$ already has the right lower semi-continuity properties, while for $\fza$ we need to explicitly construct the lower-semicontinuous envelope function $\lscfza$. 
 
Some of the other differences are only apparent. For instance, the requirement, in the definition of $\Hzb$, that for each $m^i$ the minimization problem $\fzb(m^i)$ admits a minimizer, is mirrored in two dimensions by the compactness property $\lscfza(m^i)=\fza(m^i)$. The reduction to `blobs' of bounded and separated support (Lemma~\ref{lemma:other_sequence}) is immediate in two dimensions, since it follows from the vanishing of the perimeter.

\medskip

\textsl{Minimizers of the local problem in three dimensions.}
For the three-dimensional minimization problem~\pref{def:fe3} one can show a number of properties. For instance, the concaveness of $\fzb$ for small $m$ implies that for small $m$  minimizing sequences are compact, and the minimizers are balls.  One can also show that for sufficiently large $m$, a ball with volume $m$ will be unstable with respect to symmetry-breaking perturbations; Ren and Wei have documented this phenomenon in two space dimensions~\cite{RenWei08TR}.

For the three-dimensional case, however, one can show that balls become unstable with respect to splitting into two balls of half the volume before they become unstable with respect to small symmetry-breaking perturbations. This leads us to postulate the following characterization of global minimizers, when they exist:
\begin{conjecture}
All global minimizers of the problem $\pref{def:fe3}$ are balls.
\end{conjecture}

\medskip

\textsl{Limiting structures.} In both two and three dimensions, the limiting energies `at the next level' $\Hza$ and $\Hzb$ penalize proximity of particles as if they were electrically charged. In two dimensions Lemma~\ref{lemma:charmin} guarantees that the masses $m^i$, which play the role of the charges of the particles, are all the same; in three dimensions we conjecture that the same holds, although currently we are not able to exclude the possibility of $(n-1)$ equal masses and one different mass.

The question whether minimizers of these Coulomb energies are necessarily periodic is a subtle one. It is easy to construct numerical examples of bounded domains on which minimizers \emph{can not} be periodic; see e.g.~\cite{RW6} for examples on discs in $\R^2$. At the same time, the examples with many particles do show a tendency to a triangular packing away from the boundary. In the physical literature such structures are known as \emph{Wigner crystals}, and in that field it is generally assumed that periodic structures have lowest energy. As far as we know, there are no rigorous results that show periodicity without any \emph{a priori} assumptions on the geometry.

Turning to what can be proved, the closest related result we know is the two-dimensional, Leonard-Jones crystallization result of~\cite{Th}. Moreover, for the full problem \pref{def:Fes}, the only {\it periodicity-like} results we know of, in dimension larger than one, a statement concerning the uniformity of the energy distribution on large boxes~\cite{ACO}, and for finite-size structures in $\R^n$ a scaling result bounding the energy in terms of lower-dimensional energies~\cite{vGP08}.

\medskip
\textsl{The role of the mass constraint.} Note that in the main theorems (\ref{first-order -limit}, \ref{th:sharpnextlevel}, \ref{3D-first-order -limit}, and \ref{3D-th:sharpnextlevel}) there is no mass constraint, as in~\pref{mass-constraint}, but only the weaker requirement that $\int v$ is bounded. This merits some remarks:
\begin{itemize}
\item Free minimization of the limiting functionals $\Fza$, $\Hza$, $\Fzb$, and $\Hzb$ simply yields the zero function with zero energy. In order to have a non-trivial object in the limit some additional restriction is therefore necessary. Typically one expects to have a sequence~$v_\eta$ for which the mass either is fixed or converges to a positive value.
\item The fact that only boundedness of $\int v$ is required also implies that this scaling of mass is the smallest one to give (for this scaling of the energies) non-trivial results; if $\int v$ converges to zero, then the limiting energies are also zero, and no structure can be determined. This conclusion resonates with the fact that in the formal phase diagram of the Ohta-Kawasaki functional~\pref{OK} the phase at the extreme ends of the volume fraction range is the spherical phase~\cite{CPW}.
\end{itemize}
\bigskip

\textsl{Related work. }
Our results are consistent with and complementary to two other recent studies in the regime of small volume fraction.  
In \cite{RW5} Ren and Wei prove the existence of  sphere-like solutions to the Euler-Lagrange equation of  (\ref{def:Ees}), and further investigate their stability.  They also show that the centers of sphere-like solutions are close to global minimizers of an effective energy defined over delta measures which includes both a local energy defined 
over each point measure, and a Green's function interaction term which sets their location.
While their results are similar in spirit to ours, they are based upon completely different techniques which are local rather than global.

In \cite{GC,NHR}  the authors explore the dynamics of small spherical phases for a gradient flow for  (\ref{def:Fes})  with  small volume fraction. Here one finds a separation of 
time scales for the dynamics:  Small particles both exchange material as 
in usual Ostwald ripening, and migrate because of an effectively repulsive nonlocal energetic term.  Coarsening via mass 
diffusion only occurs while particle radii are small, and they 
eventually approach a finite equilibrium size.  Migration, on the other hand, is
responsible for producing self-organized patterns. By constructing  approximations based 
upon an Ansatz of spherical particles similar to the classical LSW (Lifshitz-Slyozov-Wagner) theory, one derives a finite dimensional dynamics for particle positions and radii.
For large systems, kinetic-type equations which describe the evolution of a probability 
density are constructed.  A separation of time scales between particle growth and 
migration allows for a variational characterization of spatially inhomogeneous quasi-equilibrium states.
Heuristically this matches our findings of (a) a first order energy which is local and 
essentially driven by perimeter reduction, and (b)  a Coulomb-like interaction energy, at the 
next level,  responsible for placement and self organization of the pattern. It would be 
interesting if one could make these statements precise via the calculation of gradient 
flows and their connection with $\Gamma$-convergence~\cite{SS}. 

We further note that this  asymptotic study  has much in common with the asymptotic analysis of the well-known {\it Ginzburg-Landau functional} for the study of magnetic vortices ({\it cf.} \cite{SS2,  JS, ABO}).  
However our problem is much more direct as it pertains to the asymptotics of the support of minimizers. This is in strong contrast to the Ginzburg-Landau functional wherein one is concerned with an intrinsic vorticity quantity which is captured via a certain gauge-invariant Jacobian determinant of the order parameter.   

\medskip

Although the energy functional~\pref{OK} provides a relatively simple, and mathematically accessible, description of patterns in this block copolymer system, rigorous results characterizing minimal-energy patterns in higher dimensions are few and far between. Apart from the work in this paper we should mention the uniform energy distribution results of~\cite{ACO} which provide a weak statement of uniformity in space, and the comparison of large, localized structures in multiple dimensions with extended lower-dimensional structures~\cite{vGP08}.

For a slightly a different model for block copolymer behavior additional results are available. In~\cite{PV08TR,PV09TR,PV09TRa} the authors study an energy functional consisting of  two terms as in $\cal E$ in~\pref{def:Fes}, but with the $H^{-1}$-norm replaced by the $W^{-1,1}$-norm, or equivalently by the Wasserstein distance of order $1$. For this functional the authors study the symmetric regime, in which $A$ and $B$ appear in equal amounts; a parameter $\e$ characterizes the small length scale of the patterns. They prove that low-energy structures in two dimensions become increasingly stripe-like as $\e\to0$, that the stripe width approaches $\e$, and that the Gamma-limit of the rescaled energy measures the square of the local stripe curvature.

\bigskip

{\bf Acknowledgments:} The research of RC  was partially supported by  
an NSERC (Canada) Discovery Grant. The research of MP was partially supported by 
 NWO project 639.032.306. 
We thank Yves van Gennip for many helpful comments on previous versions of the  manuscript.

\bibliography{ref}

\begin{thebibliography}{10}   

\bibitem{A}  Alberti, G.: Variational Models for Phase Transitions, an Approach via $\Gamma$-convergence.  Calculus of Variations and Partial Differential Equations (Pisa, 1996),  95--114, Springer, Berlin, 2000. 
\bibitem{ABO} Alberti, G.,  Baldo, S., and  Orlandi, G. Variational Convergence for 
Functionals of Ginzburg-Landau Type.  Indiana Univ. Math. J.  {\bf 54-5}, 
1411--1472 (2005).
\bibitem{ACO} Alberti, G., Choksi, R., and Otto, F.: Uniform Energy Distribution for 
Minimizers of an Isoperimetric Problem With Long-range Interactions.   J. 
Amer. Math. Soc.,  {\bf 22-2},  569-605 (2009). 
\bibitem{AFP} Ambrosio, L., Fusco, N., and Pallara, D.: {\em Functions of Bounded Variation and Free Discontinuity Problems}. Oxford Science Publications 2000.  
\bibitem{AB93} G. Anzellotti and S. Baldo: Asymptotic Development by {$\Gamma$}-convergence.
Appl. Math. Optim.Ê 27Ê 105-123Ê (1993).
\bibitem{BF} Bates, F.S. and Fredrickson, G.H.: Block Copolymers - Designer Soft 
Materials. 
Physics Today, {\bf 52-2}, 
32-38 (1999).
\bibitem{Br} Braides, A.: {\em $\Gamma$-Convergence for Beginners}, Oxford Lecture Series 
in 
Mathematics and Its Applications, {\bf 22} 
2002. 
\bibitem{Bre} Brezis, H.: {\em Analyse Functionelle}. Masson, Paris 1983. 
\bibitem{Burchard96} Burchard,  A.: Cases of Equality in the Riesz Rearrangement Inequality.  Ann. of Math. (2)  {\bf 143},   no. 3, 499--527 (1996). 
\bibitem{CPW} Choksi, R., Peletier, M.A., and Williams, J.F.: On the  Phase Diagram for Microphase Separation of Diblock Copolymers: an Approach via 
a Nonlocal Cahn-Hilliard Functional.  SIAM J. Appl. Math., {\bf 69-6},  1712-1738 (2009). 
\bibitem{CP2} Choksi, R. and Peletier, M.A.: Small Volume Fraction Limit of the Diblock Copolymer Problem: II. Diffuse  Interface Functional.  In preparation. 
\bibitem{CR} Choksi, R. and Ren, X.: On a Derivation of a Density Functional 
Theory for Microphase Separation of Diblock 
Copolymers. Journal of Statistical Physics, {\bf 113} 151-176 (2003). 
\bibitem{CR2} Choksi, R. and Ren, X.:  Diblock Copolymer / Homopolymer Blends: Derivation of a Density Functional Theory.  Physica D, Vol. {\bf 203}, 100-119 (2005). 
\bibitem{vGP08} van Gennip, Y.  and  Peletier, M. A. : Copolymer-homopolymer Blends: Global Energy Minimisation and Global Energy Bounds. 
Calculus of Variations and Partial Differential Equations 33 75-111 (2008)
\bibitem{GC} Glasner, K. and Choksi, R.: Coarsening and Self-Organization in Dilute 
Diblock Copolymer Melts and Mixtures.  Physica D   {\bf 238},  1241-1255 (2009). 
\bibitem{JS} Jerrard, R.  L. and  Soner, H. M.: The Jacobian and the Ginzburg-Landau 
Energy.  Calc. Var. Partial Differential Equations  {\bf 14-2},  151-191 (2002).
\bibitem{LP}  Lachand-Robert, T. and  Peletier, M. A.:  An Example of Non-convex Minimization and an Application to Newton's Problem of the Body of Least Resistance.  Ann. Inst. H. Poincar\'e Anal. Non Lin\'eaire  {\bf 18-2}, 179--198 (2001). 
\bibitem{L1} Lions, P.L.: The Concentration-Compactness  Principle in the Calculus of Variations. The Limit Case, Part I. 
Rev. Mat. Iberoamercana {\bf 1-1}, 145 - 201 (1984). 
\bibitem{NHR} Helmers, M., Niethammer, B., and Ren, X.: {Evolution in Off-critical Diblock Copolymer Melts}.  Netw. Heterog. Media  {\bf 3-3}   615--632 (2008).
\bibitem{NO} Nishiura, Y. and Ohnishi, I.: Some Mathematical Aspects of the Micro-phase 
Separation in Diblock Copolymers. 
Physica D {\bf 84} 31-39 (1995). 
\bibitem{OK} Ohta, T. and Kawasaki, K.: Equilibrium Morphology of Block Copolymer Melts.  
Macromolecules {\bf 19} 2621-2632 (1986).
\bibitem{PV08TR}
Peletier, M. A.  and Veneroni, M.: Non-oriented Solutions of the Eikonal Equation.
Arxiv preprint arXiv:0811.3928 (2008)
\bibitem{PV09TR}
Peletier, M. A.  and Veneroni, M.: Stripe Patterns in a Model for Block Copolymers.
Arxiv preprint arXiv:0902.2611 (2009)
\bibitem{PV09TRa}
Peletier, M. A.  and  Veneroni, M.: Stripe Patterns and the Eikonal Equation.
Arxiv preprint arXiv:0904.0731 (2009)
 \bibitem{RW0} Ren, X., and Wei, J.: On the Multiplicity of Two Nonlocal Variational Problems,
SIAM J. Math. Anal. {\bf 31-4} 909-924 (2000)
\bibitem{RW4}
Ren, X. and Wei, J.: Droplet Solutions in the Diblock Copolymer Problem with Skewed 
Monomer Composition.  Calc. Var. Partial Differential Equations. {\bf   25-3}  333--359 
(2006). 
\bibitem{RW3} Ren, X. and Wei, J.:  Existence and Stability of Spherically Layered 
Solutions of the Diblock Copolymer Equation.  SIAM J. Appl. Math.  {\bf 66-3} 1080--1099 
(2006).
\bibitem{RW6} Ren, X. and Wei, J.:  Many Droplet Pattern in the Cylindrical Phase of Diblock Copolymer Morphology. 
Reviews in Mathematical Physics 19 879--921 (2007).
\bibitem{RW5} Ren, X. and Wei, J.: Spherical Solutions to a Nonlocal Free Boundary 
Problem From Diblock Copolymer Morphology. SIAM J. Math. Anal. {\bf 39-5} 1497-1535 
(2008). 
\bibitem{RenWei08TR} Ren, X. and Wei, J.: {Oval Shaped Droplet Solutions in the Saturation Process of Some Pattern Formation Problems}. Preprint (2009).
\bibitem{SS} Sandier, E. and Serfaty, S.: $\Gamma$-convergence of Gradient Flows with 
Applications to Ginzburg-Landau.  Comm. Pure Appl. Math.  {\bf 57-12}  1627--1672 (2004). 
\bibitem{SS2} Sandier, E. and Serfaty, S.: {\em Vortices in the Magnetic Ginzburg-Landau Model}. Progress in Nonlinear Differential Equations  
Vol. 70, Birkh\"{a}user 2007. 
\bibitem{SA} Seul, M. and Andelman, D.: Domain Shapes and Patterns: The Phenomenology 
 of Modulated Phases. Science {\bf 267} 476 (1995).  
\bibitem{Th} Theil, F.:  A Proof of Crystallization in Two Dimensions.  Comm. Math. Phys.  
{\bf 262-1}  209--236 (2006).
\end{thebibliography}
\bibliographystyle{plain}

\end{document}